\documentclass{amsart}
\usepackage{rotating}
\usepackage{enumerate}
\usepackage{amsfonts}
\usepackage{amsmath}
\usepackage{amssymb}
\usepackage{textcomp}
\usepackage{float,lscape}
\usepackage{mathtools}
\usepackage[titletoc,toc,title]{appendix}

\usepackage{enumitem}
\usepackage[arrow, matrix, curve, graph]{xy}
\usepackage{tikz}
\usetikzlibrary{arrows}
\usepackage[hyphens]{url}  
\newtheorem{theorem}{Theorem}[section]

\newtheorem{lemma}[theorem]{Lemma}

\newtheorem{proposition}[theorem]{Proposition}

\newtheorem{corollary}[theorem]{Corollary}

\theoremstyle{definition}

\newtheorem{definition}[theorem]{Definition}

\newtheorem{example}[theorem]{Example}

\newtheorem{remark}[theorem]{Remark}

\newtheorem *{Lemma 1}{Lemma 1}
\newtheorem *{Theorem 2}{Theorem 2}
\newtheorem *{Theorem 5}{Theorem 5}
\newtheorem *{Theorem 3}{Theorem 3}
\newtheorem *{Theorem 4}{Theorem 4}

\newtheorem *{Problem1}{The group ring isomorphism problem [GRIP]}
\newtheorem *{Problem2}{The integral group ring isomorphism problem [IGRIP]}
\newtheorem *{Problem3}{Main Problem: The twisted group ring isomorphism problem [TGRIP]}

\newcommand{\C}{{\mathbb C}}

\newcommand{\SL}{{\operatorname {SL}}}

\newcommand{\RNum}[1]{\uppercase\expandafter{\romannumeral #1\relax}}
\newcommand{\rNum}[1]{\lowercase\expandafter{\romannumeral #1\relax}}

\begin{document}
\title{Twisted group ring isomorphism problem}
\address{Departamento de matem\'aticas, Facultad de matem\'aticas, Universidad de Murcia, 30100 Murcia, Spain}
\address{Institute of Algebra and Number Theory, Pfaffenwaldring 57\\
University of Stuttgart, Stuttgart 70569, Germany}
\author{Leo Margolis}
\email{leo.margolis@um.es}
\author{Ofir Schnabel}
\email{os2519@yahoo.com}
\thanks{The first author has been supported by a Marie Curie grant from EU project 705112-ZC. The second author has been supported by the Minerva Stiftung.}

\begin{abstract}
We propose and study a variation of the classical isomorphism problem for group rings in the context of projective representations. We formulate several weaker conditions following from our notion and give all logical connections between these condition by studying concrete examples.
We introduce methods to study the problem and provide results for various classes of groups, including abelian groups, groups of central type, $p$-groups of order $p^4$ and groups of order $p^2q^2$, where $p$ and $q$ denote different primes.
\end{abstract}

\maketitle
\bibliographystyle{abbrv}
\section{Introduction}\label{Intro}\pagenumbering{arabic} \setcounter{page}{1}
For a finite group $G$ and a commutative ring $R$ the group ring $RG$ of $G$ over $R$ is a classical object of representation theory.
The question which information about $G$ is encoded in the ring structure of $RG$ has been studied by many authors,
see e.g. \cite[Chapter 12]{Isaacs} on results for $R = \mathbb{C}$ or \cite{Hertweck, RoggenkampScott} for $R = \mathbb{Z}$.
In this context so-called isomorphism problems are of particular interest.
Denote by $\Omega$ the class of finite groups, by $\Omega_n$ the groups of order $n$ and denote by $\Delta_R$ an equivalence relation on $\Omega$
which is defined by $G \Delta_R H$ if and only if $RG \cong RH$.
\begin{Problem1}
For a given commutative ring $R$, determine the equivalence classes of $\Omega$ with respect to the relation $\Delta_R$.
Answer in particular, for which groups $G\Delta_R H$ implies $G \cong H$.
\end{Problem1}
Informally, the 'in particular' part of [GRIP] asks for which groups the information contained in $RG$ determines $G$ up to isomorphism.
It is fairly easy to find rings and groups which give rise to a negative answer to the 'in particular' part of [GRIP].
For example, for any abelian groups $G$, $H$ of the same cardinality, $G \Delta_{\C} H$.
On the other hand, if $G$, $H$ are finite abelian groups and $G \Delta_{\mathbb{Q}} H$, then $G \cong H$ \cite[Theorem 3]{Perlis}.
M. Hertweck found non-isomorphic groups $G$ and $H$ such that $\mathbb{Z}G$ and $\mathbb{Z}H$ are isomorphic \cite{Hertweck}. Since for any commutative ring $R$,
\begin{equation*}
RG \cong R\otimes _{\mathbb{Z}}\mathbb{Z}G,
\end{equation*}
Hertweck's example shows the existence of groups which do not provide a positive answer for the 'in particular' part of [GRIP] over any ring $R$.
However many interesting questions regarding special classes of groups over specific rings are open. E.g. the modular isomorphism problem,
which asks whether $RG$ determines the structure of $G$ for a field $R$ of characteristic $p$ and a $p$-group $G$, or the integral isomorphism problem for groups of odd order.\\

In this paper we propose and study a 'twisted' version of [GRIP].
For a $2$-cocycle $\alpha \in Z^2(G, R^*)$ the twisted group ring $R^\alpha G$ of $G$ over $R$ with respect to $\alpha$ is the free $R$-module with basis $\{u_g\}_{g \in G}$ where the multiplication on the basis is defined via
\[u_g u_h = \alpha(g, h) u_{gh} \ \ \text{for all} \ \ g,h \in G \]
and any $u_g$ commutes with the elements of $R$.

The ring structure of $R^{\alpha}G$ depends only on the cohomology class of $\alpha$ and not on the particular $2$-cocycle.
The role twisted group rings of $G$ over $R$ play for the projective representation theory is in many ways the same played by the group ring $RG$ for the representation theory of $G$ over $R$,
as it was shown in the ground laying work of I. Schur \cite{Schur}.

For a $2$-cocycle $\alpha \in Z^2(G, R^*)$ we are going to denote by $[\alpha]\in H^2(G,R^*)$ the $2$-cohomology class of $\alpha$. We are now ready to introduce the relation of interest.
\begin{definition}
Let $R$ be a commutative ring. For $G,H \in \Omega$, $G \sim_R H$ if there exists a group isomorphism
$$\psi :H^2(G,R^*)\rightarrow H^2(H,R^*)$$
such that for any $[\alpha] \in H^2(G,R^*)$,
$$R^{\alpha}G\cong R^{\psi (\alpha)}H.$$
\end{definition}
Note that any group isomorphism $\psi: G \rightarrow H$ induces a group isomorphism $\bar{\psi}:Z^2(G,R^*)  \rightarrow Z^2(H,R^*)$ via
\[\bar{\psi}(\alpha)(x,y)=\alpha(\psi ^{-1}(x),\psi^{-1}(y)) \ \ \text{for} \ \ \alpha \in Z^2(G,R^*) \ \text{and} \ x,y \in H.\]
This in turn induces a natural ring isomorphism
    \begin{align*}
    \tilde{\psi}:R^\alpha G  & \rightarrow R^{\bar{\psi} (\alpha)}H, \\
    u_g  & \mapsto v_{\psi (g)},
    \end{align*}
and therefore the isomorphism equivalence relation of groups is a refinement of $\sim_R$. Note also that since $\psi$ maps the identity of $H^2(G, R^*)$ to the identity of $H^2(H, R^*)$ the relation $G\sim_R H$ implies $RG \cong RH$ and thus $\sim_R$ is a refinement of $\Delta_R$.

Thus we may formulate the following problem.

\begin{Problem3}
For a given commutative ring $R$, determine the equivalence classes of $\Omega$ with respect to the relation $\sim_R$.
Answer in particular, for which groups $G \sim_R H$ implies $G \cong H$.
\end{Problem3}

Consequently, another wording for the in particular part of [TGRIP] is that for a given commutative ring $R$, we wish to classify all the $\sim_R$-singletons, up to isomorphism. A related question was considered in \cite{HoffmanHumphreys}.

The difference between the twisted problem and the regular problem is demonstrated in the following lemma.
While all abelian groups of the same cardinality are $\Delta_{\mathbb{C}}$-equivalent, the situation is quite different for $\sim_{\C}$, as the following Lemma shows.
\begin{lemma}\label{lemma:abcase}
 Any abelian group $A$ is a singleton with respect to $\sim_{\C}$.
\end{lemma}
The proof of Lemma \ref{lemma:abcase} is given at the beginning of \S\ref{diagram}.

In this paper we will concentrate on the case $R = \mathbb{C}$. In this case the cohomology group $H^2(G, \mathbb{C^*})$ is often called the Schur multiplier of $G$ and we are going to denote it by $M(G)$. As proven by Schur, $M(G)$ is always a finite group \cite[Kapitel V, Hilfssatz 23.2]{Huppert}.

In a similar way to the regular case, an $\alpha$-projective representation of a group $G$ can be defined as a module over a twisted group ring $R^{\alpha}G$, where $\alpha \in Z^2(G,R^*)$.
A generalized Maschke's theorem says that for fields $F$ and finite groups $G$ such that $(\text{char}(F),|G|)=1$, twisted group algebras $F^{\alpha}G$ are semisimple for any $2$-cocycle $\alpha \in Z^2(G,F^*)$.
We remark that unlike in the ordinary case, there exist semisimple twisted group algebras $F^{\alpha}G$ where $F$ is a field and $G$ is a group such that $(\text{char}(F),|G|)>1$ (see, e.g, \cite[Theorem 3.3]{AGdR}). Let $G$ be a finite group and let $F$ be a field such that the cardinality of $G$ is prime to char$(F)$.
When a group $G$ admits a unique irreducible $\alpha$-projective representation, $\alpha$ is called nondegenerate. A group $G$ which admits a nondegenerate $\alpha \in Z^2(G, \mathbb{C}^*)$ is called of central type. The corresponding twisted group algebra $\mathbb{C}^{\alpha}G$ is simple.
Since the Artin-Wedderburn decomposition of $\mathbb{C}^\alpha G$ only depends on the cohomology class of $\alpha$, we may refer to nondegenerate cohomology classes.

We are going to solve [TGRIP] for some classes of groups. Namely we will prove
\begin{theorem}\label{th:p2q2}
Let $G$ and $H$ be groups of cardinality $n$ and let $p\neq q$ be primes. Consider the conditions
\begin{enumerate}[label=(\alph*)]
 \item $\C G\cong \C H$,
 \item $M(G) \cong M(H)$,
 \item $G$ and $H$ are not of central type.
\end{enumerate}
Then
\begin{enumerate}
 \item If $n=p^2 q^2$, then the first two of the conditions above are sufficient to determine that $G\sim _{\C} H$.
 \item If $n=p^4$ then these three conditions imply that $G\sim _{\C} H$.
\end{enumerate}
\end{theorem}
Theorem~\ref{th:p2q2}(1) is Lemma \ref{lemma:p4enough2} and Theorem~\ref{th:p2q2}(2) is Theorem~\ref{th:pqcase}.
Using Theorem~\ref{th:p2q2}, we classify in Theorem~\ref{th:p4class} the $\sim_{\C}$-classes of $\Omega_{p^4}$, where $p$ is a prime,
and demonstrate in Example~\ref{ex:319} how to classify the $\sim_{\C}$-classes of $\Omega_{p^2q^2}$ for primes $p\neq q$.
It is important to notice that in general the conditions in Theorem \ref{th:p2q2} are not sufficient to determine that two groups are $\sim _{\C}$-equivalent (see Examples \ref{Ex2} and \ref{Ex3}.).

It is clear that a group of central type can be $\sim _{\C}$-equivalent only to other groups of central type.
\S\ref{se:groupsofct} is devoted to the study of $\sim _{\C}$-equivalence classes consisting of groups of central type.
Denote the center and the commutator subgroup of $G$ by $Z(G)$ and $G'$ respectively. Note that if $G$ is a group of central type of order $n^2$ where $n$ is a square-free number then $\frac{|G|}{|G'||Z(G)|}$ is an integer (see Theorem \ref{th:squarefreedec}). We prove the following.
\begin{theorem}\label{th:centyp}
Let $G$ be a group of central type of order $n^2$ where $n$ is a square-free number. Then in both of the following cases $G$ is a $\sim _{\C}$-singleton.
\begin{enumerate}
 \item If $|G'|$ is divisible by at most two primes.
 \item If $\frac{|G|}{|G'||Z(G)|}$ is square-free.
\end{enumerate}
\end{theorem}
The case where $G$ is divisible by at most two primes is done in Proposition~\ref{prop:atmosttwoprimes}
and the case where $\frac{|G|}{|G'||Z(G)|}$ is square-free is done in Proposition~\ref{prop:cyclicaction1}.
However in Example \ref{ex:3primes} we will present two non-isomorphic groups $G$ and $H$ of order $(2\cdot 3 \cdot 5 \cdot 7)^2$ which are both of central type such that $G \sim_\mathbb{C} H$.
Here $|G'| = |H'|$ is divisible by three primes. By Theorem \ref{th:p4class} $p$-groups of central type of order at most $p^4$ are $\sim_\mathbb{C}$-singletons.
A computer aided search showed us that in fact also any group of central type of order $64$ is a $\sim_\mathbb{C}$ singleton, see Remark \ref{Order64}.
We are not aware of non-isomorphic $p$-groups $G$ and $H$ of central type such that $G \sim_\mathbb{C} H$.

There are several conditions which are necessary for $G \sim_\mathbb{C} H$ to hold. First of all since $G \sim_\mathbb{C} H$ implies an isomorphism $\psi$ between $M(G)$ and $M(H)$, a condition one might check when studying [TGRIP] is $M(G) \cong M(H)$. Furthermore, as remarked above, $\psi$ maps the identity of $M(G)$ to the identity of $M(H)$ and so $\mathbb{C}G \cong \mathbb{C} H$. Moreover $\psi$ provides a bijection between the twisted group algebras of $G$ and $H$ over $\mathbb{C}$, so that
\[\{\mathbb{C}^\alpha G \ | \ [\alpha] \in M(G)\} = \{\mathbb{C}^\beta H \ | \ [\beta] \in M(H) \}.\]
Another condition on $G$ and $H$ to fulfil $G\sim_\mathbb{C} H$ is connected to the so called Schur cover $S_G$ of $G$, often also called a representation group of $G$,
see the introduction in \cite[Kapitel V, \S \ 23]{Huppert} where also the following can be found.
$S_G$ is a group such that a group $N$ of order $|M(G)|$ is contained in $S_G' \cap Z(S_G)$ and satisfies $S_G/N \cong G$.
By the work of Schur $S_G$ always exists, although it might not be unique, and satisfies (see Theorem~\ref{SchurCover})
\[\mathbb{C}S_G \cong \oplus_{[\alpha] \in M(G)} \mathbb{C}^\alpha G. \]
So another necessary condition for $G \sim_\mathbb{C} H$ is provided by $\mathbb{C}S_G \cong \mathbb{C}S_H$. Summing up we obtain:
\begin{definition}\label{def:conditions}
We say that groups $G$ and $H$ satisfy condition A, B, C or D respectively, if
\begin{itemize}
\item[A)] $\mathbb{C}G \cong \mathbb{C}H$.
\item[B)] $M(G) \cong M(H)$.
\item[C)] $\mathbb{C}S_G \cong \mathbb{C}S_H$.
\item[D)] $\{\mathbb{C}^\alpha G \ | \ [\alpha] \in M(G)\} = \{\mathbb{C}^\beta H \ | \ [\beta] \in M(H) \}$.
\end{itemize}
\end{definition}
As stated above these conditions are necessary for $G \sim_\mathbb{C} H$.
In Section \ref{diagram} we are going to study the relations between conditions A, B, C, D and provide examples that apart from two connections between these conditions which follow from well known projective representation theory, there are no other implications between these conditions. This will lead to the following theorem.

\begin{theorem}\label{th:conditions}
Let $G$ and $H$ be finite groups and keep the notation of Definition \ref{def:conditions}. The logical connections between $\sim_\mathbb{C}$ and conditions A, B, C and D are the following.
\begin{itemize}
\item[a)] If $G \sim_\mathbb{C} H$ then $G$ and $H$ satisfy all the conditions A, B, C and D.
\item[b)] Condition $D$ implies condition A and condition C.
\item[c)] The conditions A and B combined do not imply condition C.
\item[d)] The conditions B and C combined do not imply condition A.
\item[e)] The conditions A, C and D combined do not imply condition B.
\item[f)] The conditions A, B and C combined do not imply condition D.
\end{itemize}
Summarizing: D implies A and C and there are no other logical connections betweens A, B, C and D.
\end{theorem}
The positive implications of Theorem \ref{th:conditions} are summarized in Figure \ref{Diagramm}.

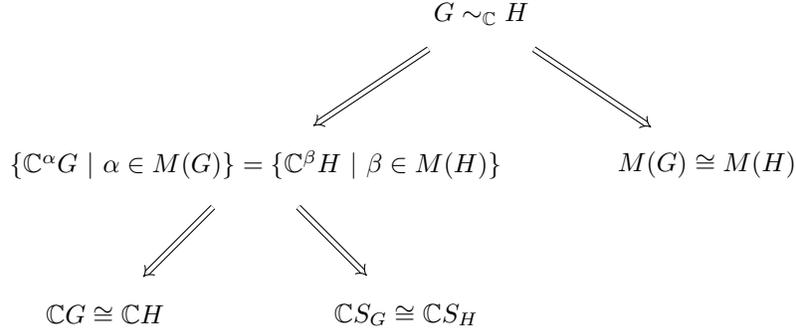
\begin{figure}[h]
\begin{center}
\begin{tikzpicture}
\tikzset{thick arc/.style={-implies,double equal sign distance, shorten <=1em, shorten >=1em, text=black}}
\tikzset{node distance=3.5cm, auto}
\node (null1)  {};
\node (relation) [right of=null1] {$G \sim_\mathbb{C} H$};
\tikzset{node distance=2cm, auto}
\node (null2) [below of=relation] {};
\tikzset{node distance=3cm, auto}
\node (projdim) [left of=null2] {$\{\mathbb{C}^\alpha G \ | \ \alpha \in M(G)\} = \{\mathbb{C}^{\beta} H \ | \ \beta \in M(H) \}$};
\node (MG) [right of=null2] {$M(G) \cong M(H)$};
\tikzset{node distance=2cm, auto}
\node (null3) [below of=projdim] {};
\node (CG) [left of=null3] {$\C G \cong \C H$};
\node (SG) [right of=null3] {$\mathbb{C}S_G \cong \mathbb{C}S_H$};

\draw[thick arc, draw=black, ] (relation) to node [near start] [left] {} (projdim);
\draw[thick arc, draw=black,] (relation) to node [near start] [left] {} (MG);
\draw[thick arc, draw=black,] (projdim) to node [near start] [right] {} (CG);
\draw[thick arc, draw=black,] (projdim) to node [near start] [right] {} (SG);

\end{tikzpicture}
\caption{Some relations connected to the twisted group ring isomorphism problem}\label{Diagramm}
\end{center}
\end{figure}

In \cite{HoffmanHumphreys} Hoffman and Humphreys constructed two non-isomorphic groups which satisfy conditions A, B and C. Moreover these group are $\sim_\mathbb{C}$-equivalent. We give many more non-isomorphic $\sim_\mathbb{C}$-equivalent groups in Theorem \ref{th:p4class} and Example \ref{ex:319}.

\section{Useful results}\label{s:pre}
In this section we collect results that we will use later to study the relation $\sim_\mathbb{C}$. We start with results about degrees of projective characters. As before $G$ will always denote a finite group.
\begin{lemma}\label{lemma:cohdiv} (\cite[Lemma 1.2(i)]{Higgs88})
The order of any cohomology class $[\alpha] \in M(G)$ divides the dimension of each $\alpha$-projective representation of $G$.
\end{lemma}
Clearly, for any group $G$ the group algebra $\C G$ admits a $1$-dimensional representation. Therefore the following is clear.
\begin{corollary}\label{cor:onedimrep}
 A twisted group algebra $\mathbb{C}^\alpha G$ admits an $\alpha$-projective representation of dimension $1$ if and only if $\alpha$ is cohomologically trivial.
\end{corollary}

As it is well known, a cohomology class $[\alpha] \in M(G)$ corresponds to an equivalence class of group extensions, see e.g. \cite[Chapter IV, § 3]{Brown2}. We will call this the extension corresponding to $[\alpha]$.
\begin{theorem}\label{ProjectiveDegrees}
\cite[Chapter 4, Theorem 1.3]{karpilovsky2} Let $[\alpha] \in M(G)$ be of order $n$ and let $G_\alpha$ be an extension of $G$ corresponding to $[\alpha]$. Then
\[\mathbb{C}G_\alpha\cong \oplus_{i = 1}^{n} \mathbb{C}^{\alpha^i}G. \]
\end{theorem}
The next result was already mentioned in the introduction
\begin{theorem}\label{SchurCover} \cite[Kapitel V, Satz 23.8]{Huppert} If $S_G$ is a Schur cover of $G$, then
\[\mathbb{C}S_G \cong \oplus_{\alpha \in M(G)} \mathbb{C}^\alpha G. \]
\end{theorem}

We now collect results about cohomology of semi-direct products which will be useful for the calculations of Schur multipliers.
Let $G=N\rtimes T$ and denote by $N^*$ the group of one-dimensional linear representations of $N$. Then $T$ acts on $N^*$ by setting
\begin{equation}\label{eq:diagaction}
\chi^t(n) = \chi(n^{t^{-1}}) \quad \text{for} \quad n \in N, \chi \in N^*, t \in T.
\end{equation}
This action is sometimes called diagonal action. Also, the action of $T$ on $N$ induces a natural action of $T$ on $M(N)$ (see \cite[\S\ 3.3]{ginosargradings}).
With respect to this action, denote by $M(N)^T$ the $T$-invariant cohomology classes in $M(N)$.
\begin{lemma}\label{eq:cohosemi}(\cite[Corollary 2.2.6]{karpilovsky1})
With the above notations, if the cardinalities of $N$ and $T$ are coprime then
\begin{equation*}
 M(G) = M(N)^T \times M(T).
\end{equation*}
\end{lemma}
The action of $T$ on $N$ induces also a map $d: M(N)^T \rightarrow H^2(T, N^*)$ which is defined as follows.
For $[\alpha] \in M(N)^T$ there is a map $u: T \times N \rightarrow \mathbb{C}^*$ such that $u(1,n) = u(t,1) = 1$ for all $n \in N, t \in T$ and moreover
\[\alpha(n, n')\alpha(n^t, n'^t)^{-1} = u(t,n)u(t,n')u(t,nn')^{-1} \ \ \text{for} \ \ n,n ' \in N, \ t \in T. \]
Then
\[\alpha_*(t,t')(n) = u(t', n)u(tt', n)^{-1}u(t,n^{t'}) \ \ \text{for} \ \ n \in N, \ t, t' \in T \]
is a $2$-cocycle in $Z^2(T, N^*)$ and $d([\alpha])$ is the comohology class of $\alpha_*$.
\begin{theorem}\label{TaharaSchurMultiplierCalculation} (Tahara '72, see \cite[Theorem 2.2.5]{karpilovsky1})
Let $G$ be a finite group such that $G$ is a semidirect product of the form $G = N \rtimes T$.  Then $M(G) \cong \tilde{M}(G) \times M(T)$ and there is an exact sequence
\[1 \rightarrow H^1(T, N^*) \rightarrow \tilde{M}(G) \rightarrow M(N)^T \xrightarrow{d} H^2(T, N^*), \]
where  $\tilde{M}(G)$ denotes the kernel of the natural map $M(G) \rightarrow M(T)$.
\end{theorem}
For the next theorem we use the following setup.
Let $G$ be a finite group and let $M$ be a $G$-module.
Consider the trace map $\operatorname{Tr}:M\rightarrow M$ defined by $m\mapsto \sum_{g\in G} m^g$ and denote by $M^G$ the elements in $M$ which are invariant with respect to the $G$-action.
Then, by definition
$$H^0(G,M)\cong M^G/\text{Im}(Tr(M)).$$
Now, since cyclic groups have a periodic cohomology with period 2 (see e.g \cite[Chapter 6, Section 9]{Brown2}), for cyclic $G$
$$H^2(G,M)\cong M^G/\text{Im}(Tr(M)).$$

\begin{theorem}\label{H1H2TN*}
Let $T$ be a group acting on another group $N$.
\begin{enumerate}
 \item \cite[Kapitel I, Satz 17.3]{Huppert} The order of $H^1(T, N^*)$ is the number of conjugacy classes of complements of $N^*$ in $N^* \rtimes T$.
 \item For $\chi \in N^*$ set $\operatorname{Tr}(\chi) = \sum_{t \in T} \chi^t.$ If $T$ is cyclic then
\[H^2(T, N^*) \cong (N^*)^T/\operatorname{Im}(\operatorname{Tr}(N^*)).\]
\end{enumerate}
\end{theorem}

We will also use the following results on degrees of ordinary characters.
\begin{theorem}\label{OrdinaryDegreeCalculations} Let $G$ be a finite $p$-group.
\begin{enumerate}
\item[a)] \cite[Theorem 3.6]{Isaacs} If $\chi$ is an irreducible complex character of $G$ then $\chi(1)^2$ divides $|G:Z(G)|$.
\item[b)] \cite[Theorem 12.11]{Isaacs} The degrees of irreducible complex characters of $G$ are exactly $1$ and $p$ if and only if $G$ possesses an abelian normal subgroup $A$ such that $|G:A| = p$ or $|G:Z(G)| = p^3$.
\end{enumerate}
\end{theorem}

A few times we will use the notion of $\alpha$-regularity of elements in $G$ with respect to $\alpha \in Z^2(G,\C ^*)$ and its connection
to the center $\C ^{\alpha} G$.
\begin{definition}\label{def:f-reg}
Let $G$ be a finite group and let $\alpha \in Z^2(G,\C)$.
An element $g\in G$ is $\alpha$-regular if for any $h\in C_G(g)$, $[u_g,u_h]=1$ in $\C ^{\alpha}G$.
\end{definition}
It is easy to show that $\alpha$-regularity is a class property, both conjugacy and cohomology.
Let $\C ^{\alpha}G$ be a twisted group algebra, let $x\in G$ be $\alpha$-regular and fix a left transversal $T$
of $C_G(x)$ in $G$. Then,
\begin{equation*}
S_{(1 ,x)}=\sum_{t\in T} u_{t}u_xu_{t}^{-1}
\end{equation*}
is central in $\C ^{\alpha}G$ (see \cite[Proposition 6,2]{ginosar2012semi}).
By \cite[Theorem 2.4]{Oystaeyen} the center of the twisted group algebra $\C ^{\alpha}G$ has a basis consisting of
the elements $S_{(1 ,x)}$ where $[x]$ runs over the $\alpha$-regular conjugacy classes in $G$. In particular, if $\alpha$ is cohomologically trivial, then all the elements in $G$ are $\alpha$-regular.

\section{Examples and conditions}\label{diagram}
In this part we present some preliminary results on [TGRIP] over the complex numbers
and study the connection of the different conditions A, B, C and D following from $G \sim_\mathbb{C} H$ introduced in Definition \ref{def:conditions}. Put together these examples imply Theorem \ref{th:conditions}.

First we want to show that unlike the ordinary group algebra case, here the structure of an abelian group $A$ is determined by
the complex twisted group algebras over $A$, i.e. we give a proof of Lemma \ref{lemma:abcase}.
Properties of cocycles and cohomology classes can be expressed in terms of the
associated twisted group algebras. This is essentially useful in the investigation of cohomology of abelian groups.
In this case, when given an abelian group $G$ and a cohomology class $[\alpha]\in M(G)$ with an associated twisted group algebra
$\mathbb{C}^\alpha G$ spanned by $\{u_g\}_{g\in G}$ then $[\alpha]$ is uniquely determined by the values in $\mathbb{C}^*$ of all the commutators
$[u_{g_i},u_{g_j}]$, $g_i,g_j \in G$. Clearly, the converse is also true, that is the commutator relations are uniquely determined by $[\alpha]$.
In fact, if
\begin{equation}\label{eq:decompofabeliangroupptok}
A=C_{n_1}\times C_{n_2}\times \ldots \times C_{n_k}\cong \langle x_1 \rangle \times \langle x_2 \rangle\times \ldots \times \langle x_k \rangle
\end{equation}
such that $n_i$ is a divisor of $n_{i+1}$ for any $1\leq i\leq k-1$, then $M(A)$ is an abelian group of rank
$\frac{k(k-1)}{2}$
generated by the tuple of functions
$$\left(\alpha _{ij}\right)_{1\leq i<j\leq k},$$
where $\alpha _{ij}(x_i,x_j)$ is an $n_i$-th root of unity and $\alpha_{ij}$ has the value 1 on the other generators of $A$ elsewhere (see \cite{Schur}).
As a consequence of the above discussion non-isomorphic abelian groups of the same cardinality admit non-isomorphic cohomology groups.
This completes the proof of Lemma \ref{lemma:abcase}.

From the above discussion the following is clear.
\begin{corollary}\label{cor:SchurofCn}
 Let $n$ be a natural number. The Schur multiplier $M(C_n)$ is trivial and $M(C_n\times C_n)\cong C_n$.
\end{corollary}

Lemma \ref{lemma:abcase} emphasizes the difference between the regular problem [GRIP] and the twisted problem [TGRIP].
Indeed, with respect to $\sim_{\C}$ all the abelian groups are singletons while all the abelian groups of the same cardinality belong
to the same equivalence class with respect to $\Delta_{\C}$.
Lemma \ref{lemma:abcase} does however not capture the full complexity of [TGRIP].
Indeed, in order to show that abelian groups are singletons in $\sim_{\C}$ all we did is to show that abelian groups of the same order admit non-isomorphic cohomology groups. In order to better emphasize the complexity of [TGRIP] we will study the relations between conditions A, B, C and D, defined in Definition \ref{def:conditions}, in several examples and establish their interaction described in Theorem \ref{th:conditions}.

By Corollary~\ref{cor:onedimrep}, condition D implies condition A. Moreover by Theorem \ref{SchurCover} we know that condition D implies condition C. We give several examples showing the independence of the other properties.

\begin{example}\label{Ex1} We provide an example that conditions A and B combined do not imply C. \\
Let $G \cong Q_8 \times C_2$ and
\[H \cong (C_4 \times C_2) \rtimes C_2  = (\langle a \rangle \times \langle b \rangle) \rtimes \langle c \rangle, \ \ \text{where} \ \ a^c = ab, \ b^c = b.\]
The group $H$ is of central type, while $G$ is not \cite[Theorem 3.9, Lemma 3.10]{Schnabelnew}, so C does not hold. Since factoring out the central involution of the quaternion factor of $G$ we obtain an abelian group and also $H/\langle b \rangle$ is abelian, we have $|G/G'| = |H/H'| = 8,$ so $\mathbb{C}G$ and $\mathbb{C}H$ have exactly $8$ direct summands of dimension $1$. This implies
\[\mathbb{C}G \cong \mathbb{C}H \cong 8\mathbb{C} \oplus 2\mathbb{C}^{2 \times 2}, \]
since the order of $G$ and $H$ is just $16$. So $G$ and $H$ satisfy condition A.

Moreover the Schur multipliers of groups of order $16$ are well known (see e.g. \cite[Table I]{MR913203}), in particular
\[M(G) \cong M(H) \cong C_2 \times C_2, \]
so $G$ and $H$ satisfy condition B.
\end{example}

\begin{example}\label{Ex2} We next provide an example that B and C combined do not imply A.\\
Let
\[S \cong (C_8 \rtimes C_4) \rtimes C_2 = (\langle x \rangle \rtimes \langle y \rangle) \rtimes \langle z \rangle, \ \text{where} \ x^y = x^5, \ x^z = xy^3, \ y^z = y^3.\]
Then $Z(S) = \langle x^4, y^2 \rangle$, while $S' = \langle x^4, y \rangle$.
Set
\[G = S/\langle y^2 \rangle \cong (C_8 \rtimes C_2) \rtimes C_2 = (\langle a\rangle \rtimes \langle b\rangle ) \rtimes \langle c \rangle, \ \text{with} \ a^b = a^5, \ b^c = b, \ a^c = ab \]
where $a$, $b$ and $c$ denote the images of $x$, $y$ and $z$ respectively. Moreover set
\[H = S/\langle x^4 \rangle \cong (C_4 \times C_4) \rtimes C_2 = (\langle r \rangle \times \langle s \rangle ) \rtimes \langle t \rangle, \ \text{with} \ r^t = rs^3, \ s^t = s^3, \]
where $r$, $s$ and $t$ denote the images of $x$, $y$ and $z$ respectively.

\begin{itemize}[leftmargin=5pt]
\item Condition B holds: Namely we will show that
\[M(G) \cong M(H) \cong C_2\]
using Theorem \ref{TaharaSchurMultiplierCalculation}. We start with $G$. Set $N = \langle a, b \rangle$ and $T = \langle c \rangle$.
Then by Corollary~\ref{cor:SchurofCn}, $M(T) = 1$. By \cite[Table 1]{MR913203} $M(N) = 1$ and therefore , also $M(N)^T$ is trivial.
So to apply Theorem~\ref{TaharaSchurMultiplierCalculation} it will be enough to show that $H^1(T, N^*)$ has order $2$,
where $N^*$ denotes the abelian group of $1$-dimensional complex characters of $N$ on which $T$ acts as in~\eqref{eq:diagaction}.
By Theorem \ref{H1H2TN*} this is equivalent to showing that up to conjugation $N^*$ has exactly two complements in $N^* \rtimes \langle c \rangle$.  We have
\[N^* \rtimes \langle c \rangle \cong (C_4 \times C_2) \rtimes \langle c \rangle = (\langle \sigma \rangle \times \langle \tau \rangle) \rtimes \langle c \rangle \ \text{where} \ \sigma^c = \sigma, \ \tau^c = \sigma^2 \tau.\]
Here $\sigma$ corresponds to the character mapping $a$ to $i$ and $b$ to $1$, while $\tau$ corresponds to the character mapping $a$ to $1$ and $b$ to $-1$.
The complements of $N^*$ in this group are $\langle c \rangle, \langle \sigma^2 c \rangle, \langle c\sigma\tau \rangle$ and $\langle c \sigma^3 \tau \rangle$
and these form two conjugacy classes (namely $c$ is conjugate to $\sigma^2 c$). So $M(G) \cong C_2$.

We proceed to calculate $M(H)$. Set $N = \langle r, s \rangle$ and $T = \langle t \rangle$.
Let $\sigma, \tau \in N^*$ such that $\sigma$ maps $r$ to $i$ while mapping $s$ to $1$ and $\tau$ maps $r$ to $1$ while mapping $s$ to $i$. Then
\[N^* \rtimes \langle t \rangle \cong (C_4 \times C_4) \rtimes \langle t \rangle = (\langle \sigma \rangle \times \langle \tau \rangle) \rtimes \langle t \rangle, \
\text{where} \ \sigma^t = \sigma, \ \tau^t = \sigma^3 \tau^3. \]
Then one computes that there are four complements of $N^*$ in $N^* \rtimes \langle t \rangle$ all being conjugate to $\langle t \rangle$.
Hence $H^1(T, N^*) = 1$. Furthermore $(N^*)^T = \langle \sigma \rangle = \operatorname{Im}(\operatorname{Tr}(N^*))$, so $H^2(T, N^*) = 1$ by Theorem \ref{H1H2TN*}.
By Corollary~\ref{cor:SchurofCn} $M(N) \cong C_4$. Assume that in the twisted group algebra $\mathbb{C}^\alpha N$ we have $[u_r, u_s] = \lambda$.
A simple calculation shows that $[u_{r^t}, u_{s^t}] = \lambda^7$. Therefore a class $[\beta] \in M(N)$  such that $[u_r, u_s] = i$ in $\mathbb{C}^\beta N$ is not invariant under $t$,
while the class $[\gamma]$ such that $[u_r, u_s] = -1$ in $\mathbb{C}^\gamma N$ is invariant.
Hence $M(N)^T \cong C_2$ and altogether by Theorem \ref{TaharaSchurMultiplierCalculation} we have an exact sequence
\[1 \rightarrow M(H) \rightarrow C_2 \rightarrow 1 ,\]
so $M(H) \cong C_2$.  So $G$ and $H$ satisfy condition B.

\item Condition C holds: By the above $M(G) \cong M(H) \cong C_2$. Since $G$ is a quotient of $S$ by a normal subgroup $N$ of the same order as $M(G)$ and $N$ is contained in $S' \cap Z(S)$, the group $S$ is by definition a Schur cover of $G$. The same argumentation shows that $S$ is a Schur cover of $H$ and in particular
\[\mathbb{C}S_G \cong \mathbb{C}S \cong \mathbb{C}S_H, \]
so condition C also holds.

\item Condition A does not hold: We have $Z(G) = \langle a^4 \rangle$, so $|G:Z(G)| = 16$,
and also no maximal subgroup of $G$ is abelian. This implies by Theorem \ref{OrdinaryDegreeCalculations}b) that $G$ has an irreducible complex representation of degree at least $4$.

On the other hand $Z(H) = \langle r^2s \rangle$, so $|H:Z(H)| = 8$ and hence by Theorem \ref{OrdinaryDegreeCalculations}b) the maximal degree of an irreducible complex representation of $H$ is $2$.
Thus $G$ and $H$ do not satisfy condition A.
\end{itemize}
\end{example}

\begin{example}\label{Ex3}
We next provide an example for A, C and D combined together not implying B.\\
Let
\begin{align*}
 S &\cong ((C_{16} \rtimes C_4) \times C_2) \rtimes C_2 = ((\langle x \rangle \rtimes \langle y \rangle) \times \langle z \rangle) \rtimes \langle w \rangle, \\
 &\text{where} \ x^y = x^{13}, \ x^w = xy^2, \ y^w = x^8yz, \ z^w = z.
\end{align*}
 Then $Z(S) = \langle x^4, z \rangle$, while $S' = \langle x^4, y^2, z \rangle$.
Set
\[G = S/\langle x^8, z \rangle \cong (C_8 \rtimes C_4) \rtimes C_2 = (\langle a \rangle \rtimes \langle b \rangle ) \rtimes \langle c \rangle, \ \text{with} \ a^b = a^5, \ b^c = b, \ a^c = ab^2,\]
where $a$, $b$ and $c$ denote the images of $x$, $y$ and $w$ respectively.
 Moreover let
\[H = S/\langle x^4z \rangle \cong (C_8 \times C_4) \rtimes C_2 = (\langle r \rangle \times \langle s \rangle ) \rtimes \langle t \rangle, \ \text{with} \ r^t = rs^2, \ s^t = r^4s,\]
where $r$, $s$ and $t$ denote the images of $xz$, $y$ and $w$ respectively.
\begin{itemize}[leftmargin=5pt]
\item Condition B does not hold: We claim that
\[M(G) \cong C_2 \times C_2 \]
and we are going to use Theorem \ref{TaharaSchurMultiplierCalculation} to prove this. Set $N = \langle a, b \rangle$ and $T = \langle c \rangle$. By Corollary~\ref{cor:SchurofCn} $M(T) = 1$.
Therefore $M(G)\cong \tilde{M}(G)$. Moreover
\[N^* \rtimes \langle c \rangle = (\langle \sigma \rangle \times \langle \tau \rangle) \rtimes \langle c \rangle \cong (C_4 \times C_4) \rtimes C_2, \ \text{where} \ \sigma^c = \sigma \tau^2, \ \tau^c = \tau.\]
Here $\sigma$ denotes the representation mapping $a$ to $i$ while mapping $b$ to $1$ and $\tau$ denotes the representation mapping $a$ to $1$ while mapping $b$ to $i$.
In this group the conjugacy classes of complements of $N^*$ are represented by $\langle c \rangle$ and $\langle \sigma^2 c \rangle$, so $H^1(T, N^*) \cong C_2$.
Define
\[\alpha: N \times N \rightarrow \mathbb{C}^*, \ \ \alpha(a^k b^l, a^m b^n) = i^{kn}.\]
One easily checks that this defines a $2$-cocycle on $N$. In the corresponding twisted group algebra one obtains the relation $[u_a, u_b] = iu_a^4$.
Thus $u_{a^2}$ does not commute with $u_b$ and $a^2$ is not an $\alpha$-regular class. Thus the twisted group algebra $\mathbb{C}^\alpha N$ is not isomorphic to $\C N$, showing that the cohomology class of $\alpha$ is not trivial. Up to coboundaries $\alpha$ is the only non-trivial twist on $N$, as can be easily computed or also obtained from  \cite{Tahara}. The cohomology class of $\alpha$ on $N$ is also invariant under the action of $T$, so $M(N)^T \cong C_2$. Hence from Theorem \ref{TaharaSchurMultiplierCalculation} we obtain an exact sequence
\[1 \rightarrow C_2 \rightarrow M(G) \rightarrow C_2 \xrightarrow{d} H^2(T, N^*). \]
We need to show that $d$ is the trivial map. This would provide $|M(G)| = 4$. Since $S$ is a central extension of $G$ by $C_2 \times C_2$ we know that $C_2 \times C_2 \leq M(G)$ and thus $|M(G)| = 4$ implies $M(G) \cong C_2 \times C_2$. Set $u: T \times N \rightarrow \mathbb{C}^*$ by defining $u(1,n) = 1$ for all $n \in N$ and
\[u(c,a^m b^n) = \left\{ \begin{array}{ll} 1, & m \equiv 0,1,4,5 \mod 8 \\  -1, & m \equiv 2,3,6,7 \mod 8. \end{array}\right. \]
One may check that this $u$ satisfies the definition of $u$ in Theorem \ref{TaharaSchurMultiplierCalculation} with respect to $\alpha$.
Thus $u(t, n) = u(t,n^c)$ for all $t \in T$ and $n \in N$ and since $u$ only takes the values $\pm 1$ we get that the corresponding $\alpha_*$
in Theorem \ref{TaharaSchurMultiplierCalculation} is the trivial map. Thus $d([\alpha])$ is the trivial class.

Next we show
\[M(H) \cong C_4, \]
again by Theorem \ref{TaharaSchurMultiplierCalculation}. Set $N = \langle r, s \rangle$ and $T = \langle t \rangle$.
Then as above we find $H^1(T, N^*) \cong C_2$. Now, by Corollary~\ref{cor:SchurofCn} $M(N) \cong C_4$, but the generator of $M(N)$ is not invariant under $T$,
while the square of the generator is, so $M(N)^T \cong C_2$. A $2$-cocycle realizing this $T$-invariant cohomology class is given by
\[\alpha: N \times N \rightarrow \mathbb{C}^*, \ \ \alpha(r^k s^l, r^m s^n) = (-1)^{kn}.\]
Again we need to show that $d$ is the trivial map. With the $\alpha$ above we have $\alpha(n,n')\alpha(n^t,n'^t)^{-1} = 1$ for all $n,n ' \in N$.
So defining $u$ as the function being constantly $1$ one obtains a $u$ satisfying the definition in Theorem \ref{TaharaSchurMultiplierCalculation}. Thus $d$ is trivial and $M(H) \cong C_4$.
In particular $G$ and $H$ do not satisfy condition B.

\item Condition C holds: From the Schur multipliers of $G$ and $H$ computed above and the construction of $G$ and $H$ from $S$ as quotients of $S$ by normal subgroups in $S'\cap Z(S)$ we know that $S$ is a Schur cover of both $G$ and $H$ and so
\[\mathbb{C}S_G \cong \mathbb{C}S_H.\]

\item Condition A holds: We have $G' = \langle a^4, b^2 \rangle$ and $Z(G) = \langle a^2, b^2 \rangle$, while $H' = \langle r^4, s^2 \rangle$ and $Z(H) =  \langle r^2, s^2 \rangle$.
So in particular the maximal degree of an irreducible representation of $G$ and $H$ is $2$ by Theorem \ref{OrdinaryDegreeCalculations}b) and all together
\[\mathbb{C}G \cong \mathbb{C}H \cong 16\mathbb{C} \oplus 12\mathbb{C}^{2\times 2}.\]

\item Condition D holds: Let $M(H)$ be generated by an element $[\alpha]$. Set $K = S/\langle x^8 \rangle$.
Then $K$ is a central extension of degree $2$ of both $G$ and $H$. Denote by $[\beta]$ an element of $M(G)$ corresponding to $K$. Then by Theorem \ref{ProjectiveDegrees} we have
\[\mathbb{C}K \cong \mathbb{C}H \oplus \mathbb{C}^{\alpha^2} H \cong \mathbb{C}G \oplus \mathbb{C}^\beta G. \]
Consequently, since $\C H\cong \C G$ we get $\mathbb{C}^{\alpha^2} H \cong \mathbb{C}^\beta G$.
Denote by $[\gamma]$ and $[\delta]$ the other non-trivial elements in $M(G)$. By Theorem \ref{SchurCover}, since $S$ is a Schur cover of both $G$ and $H$ we obtain
\[\mathbb{C}S \cong \mathbb{C}H \oplus \mathbb{C}^\alpha H  \oplus \mathbb{C}^{\alpha^2} H \oplus \mathbb{C}^{\alpha^3} H  \cong \mathbb{C}G \oplus \mathbb{C}^\beta G \oplus \mathbb{C}^\gamma G \oplus \mathbb{C}^\delta G, \]
so by the above
\[\mathbb{C}^\alpha H  \oplus \mathbb{C}^{\alpha^3}H \cong \mathbb{C}^\gamma G \oplus \mathbb{C}^\delta G. \]
The degree of an $\alpha$-projective representation of $H$ is divisible by $4$ by Lemma \ref{lemma:cohdiv}.
But since $|S:Z(S)| = 32$ we know by Theorem \ref{OrdinaryDegreeCalculations}a) that $S$ possesses no irreducible representation of degree $8$. Hence
\[4\mathbb{C}^{4 \times 4} \cong \mathbb{C}^\alpha H \cong \mathbb{C}^{\alpha^3} H  \cong \mathbb{C}^\gamma G \cong \mathbb{C}^\delta G, \]
so overall condition D holds for $G$ and $H$.\\
\end{itemize}
\end{example}

\begin{example}\label{Ex4} We provide an example for A, B and C combined together not implying D.\\
Let
\begin{align*}
S &= ((\langle x \rangle \rtimes \langle y \rangle) \times \langle z \rangle) \rtimes \langle w \rangle \cong ((C_8 \rtimes C_8) \times C_2) \rtimes C_2, \\
&\text{where} \ x^y = x^{-1}, \ x^w = xy^4z, \ y^w = y^5, \ z^w = z.
 \end{align*}
Then $S$ is a group of order $256$. We have $Z(S) = \langle x^4, y^2, z \rangle \cong C_4 \times C_2 \times C_2$. Moreover $S' = \langle x^2, y^4, z \rangle \cong C_4 \times C_2 \times C_2$. Set
 \[G = S/\langle x^4, z \rangle \]
 and
 \[H = S/\langle x^4, y^4z \rangle. \]
 Then
 \[ G = (\langle a \rangle \rtimes \langle b \rangle) \rtimes \langle c \rangle \cong (C_4 \rtimes C_8) \rtimes C_2, \ \text{where} \ a^b = a^{-1}, \ a^c = ab^4, \ b^c = b^5,\]
 where $a$, $b$ and $c$ denote the images of $x$, $y$ and $w$ respectively and
  \[H = (\langle r \rangle \rtimes \langle s \rangle) \rtimes \langle t \rangle \cong (C_4 \rtimes C_8) \rtimes C_2, \ \text{where} \ r^s = r^{-1}, \ r^t = r, s^t = s^5,\]
  where $r$, $s$ and $t$ denote the images of $x$, $y$ and $w$ respectively.
\begin{itemize}[leftmargin=5pt]
\item Condition A holds: We have
\[|G:G'| = |H:H'| = 16 \ \text{and} \ |G:Z(G)| = |H:Z(H)| = 8.\]
 So by Theorem \ref{OrdinaryDegreeCalculations}b)
 \[\mathbb{C}G \cong \mathbb{C}H \cong 16\mathbb{C} \oplus 12\mathbb{C}^{2\times 2}. \]
\item Condition B holds: We claim that
 \[M(G) \cong M(H) \cong C_2 \times C_2 \]
  and this will again be shown using Theorem \ref{TaharaSchurMultiplierCalculation}. We start with $G$. Set $N = \langle a, b\rangle$ and $T = \langle c \rangle$.
  Computing as in Examples \ref{Ex2} and \ref{Ex3} we find $H^1(T, N^*) \cong C_2$. One may check that $M(N) \cong C_2$ and that the map $\alpha: N \times N \rightarrow \mathbb{C}^*$ defined by
\[\alpha(a^k b^l,a^m b^n) = \left\{ \begin{array}{ll} 1, & n \ \ \text{even} \\  i^{k + 2kl}, & n \ \ \text{odd} \end{array}\right. \]
is a $2$-cocycle which is not cohomologues to the trivial $2$-cocycle. One checks that $[\alpha]$ is $T$-invariant and moreover $\alpha(n,n')\alpha(n^t,n'^t)^{-1} = 1$ for all $n,n ' \in N$. So one may choose $u$ in Theorem \ref{TaharaSchurMultiplierCalculation} to be the constant $1$-function and thus $d$ is the trivial map. This shows that $|M(G)| = 4$ and since $S$ is a central extension of $G$ by $C_2 \times C_2$, implying $C_2 \times C_2 \leq M(G)$, we obtain $M(G) \cong C_2 \times C_2$. In a similar way one can show that $M(H) \cong C_2 \times C_2$.
\item Condition C holds: By the Schur multipliers of $G$ and $H$ and their construction from $S$ it follows that $S$ is a Schur cover of both $G$ and $H$ and
 \[\mathbb{C}S_G \cong \mathbb{C}S_H. \]
\item Condition D does not hold: This will follow from the fact that there is only one $[\alpha] \in M(G)$ such that $\mathbb{C}^\alpha G$ is a direct sum of $2 \times 2$-matrice rings, but there are two $[\beta] \in M(H)$ such that $\mathbb{C}^\beta H$ is such a sum.

 This follows from Theorem \ref{OrdinaryDegreeCalculations}b): Let $v$ be an element running through the set $\{x^4, z, x^4z, y^4z, x^4y^4z \}$ in $S$. Note that $\{\mathbb{C}^\alpha G |[\alpha] \in M(G) \}$ is given by the complex group algebras of the groups $G$, $S/\langle x^4 \rangle$, $S/\langle z \rangle$ and $S/\langle x^4z \rangle$, while the algebras in $\{\mathbb{C}^\alpha H |[\alpha] \in M(H) \}$ are given by the complex group algebras of $H$, $S/\langle x^4 \rangle$, $S/\langle y^4z \rangle$ and $S/\langle x^4y^4z \rangle$. Now $S/\langle v \rangle$ has a maximal subgroup which is abelian if and only if $v = y^4z$. Moreover
 \[|S/\langle v \rangle : Z(S/\langle v \rangle )| = 16\]
  if $v \neq x^4$ and
  \[|S/\langle x^4 \rangle : Z(S/\langle x^4 \rangle)| = 8.\]

 Thus by Theorem \ref{OrdinaryDegreeCalculations}b) $M(G)$ contains exactly one class $[\alpha]$ such that $\mathbb{C}^\alpha G$ is a direct sum of $2\times 2$-matrix rings, namely the class corresponding to the extension $S/\langle x^4 \rangle$, while $M(H)$ contains two such classes, the one corresponding to $S/\langle x^4 \rangle$ and the one corresponding to $S/\langle y^4z \rangle$. In particular it is not possible to find a bijection between the multisets $\{\mathbb{C}^\alpha G \ | \ [\alpha] \in M(G)\}$ and $\{\mathbb{C}^\beta H \ | \ [\beta] \in M(H) \}$. So condition D does not hold.
\end{itemize}
\end{example}

\section{Groups of order $p^4$ and $p^2q^2$}\label{se:card}
From the definition of $\sim_{\C}$ it is clear that if $G \sim_{\C} H$ then $\C G \cong \C H$ and\\ $M(G)\cong  M(H)$.
By Examples~\ref{Ex2} and \ref{Ex3}, these two conditions on the groups $G$ and $H$ are not sufficient to determine whether $G \sim_{\C} H$.
Another example for this is provided in Example \ref{ex:3primes}.
However, for groups of certain orders these conditions together with the knowledge whether $G$ and $H$ are of central type or not, are sufficient to determine whether $G\sim_{\C} H$ holds or not.

As an immediate consequence from Lemma \ref{lemma:cohdiv} we get.
\begin{corollary}\label{cor:nontrivnonnd}
Let $G$ be a group of cardinality $p^4$ for a prime $p$, and let $[\alpha]\in M(G)$ be a non-trivial degenerate cohomology class. Then,
\begin{equation*}
\C ^\alpha G\cong \oplus _{i=1}^{p^2} \C^{p \times p}.
\end{equation*}
\end{corollary}
Therefore, the following is also clear.
\begin{lemma}\label{lemma:p4enough2}
 Let $p$ be prime and let $G$ and $H$ be groups of order $p^4$ with the following properties
 \begin{enumerate}
  \item $G$, $H$ are both not of central type.
  \item $M(G) \cong M(H)$.
  \item $\C G \cong \C H$.
  \end{enumerate}
Then, $G \sim _{\C} H$.
\end{lemma}
\begin{proof}
 Let $\psi: M(G) \rightarrow M(H)$ be an isomorphism.
Then, since $\psi$ is an isomorphism it takes the identity of $M(G)$ to the identity of $M(H)$.
Consequently,
\begin{equation}\label{eq:trivcong}
 \C^{[1]}G \cong \C G \cong \C H \cong \C^{\psi([1])} H.
\end{equation}
Since both $G$ and $H$ are not of central type, any cohomology class in $M(G)$ and in $M(H)$ is degenerate.
Therefore, by Corollary~\ref{cor:nontrivnonnd}, for any non-trivial cohomology class $[\alpha]\in M(G)$,
\begin{equation}\label{eq:nontrivcong}
 \C^{\alpha}G \cong \C^{\psi(\alpha)} H.
\end{equation}
By~\eqref{eq:trivcong} and~\eqref{eq:nontrivcong}, $G \sim _{\C} H$.
\end{proof}
As a result of the above, Lemma~\ref{lemma:abcase} and by \cite[Table 2, Table 3]{Higgs2006} we get the partition of $\Omega _{p^4}$ to $\sim_{\C}$-equivalence classes for any prime $p$.
We use the notations of Table~\ref{tab:p4} and Table~\ref{tab:16} which can be found in the appendix.
\begin{theorem}\label{th:p4class}
 Let $p$ be an odd prime. With the notations of Table~\ref{tab:p4}, in the partition of $\Omega _{p^4}$ to $\sim_{\C}$-equivalence classes, all groups are singletons
 except for the two size three equivalence classes $\{G_{(\text{vii})},G_{(\text{ix})},G_{(\text{x})}\}$ and $\{G_{(\text{xi})},G_{(\text{xii})},G_{(\text{xiii})}\}$.
 With the notations of Table~\ref{tab:16}, in the partition of $\Omega _{16}$ to $\sim_{\C}$-equivalence classes  all groups are singletons
 except for the two size two equivalence classes $\{G_{(\text{vii})},G_{(\text{xi})}\}$ and $\{G_{(\text{xiii})},G_{(\text{xiv})}\}$.
\end{theorem}
\begin{proof}
 By Lemma~\ref{lemma:abcase} the abelian groups are all $\sim_{\C}$-singletons.
 Let $p$ be an odd prime. By \cite[Theorem 1]{Schnabelnew} or by \cite[Table 2]{Higgs2006} the groups $\{G_{(\text{viii})},G_{(\text{xiv})},G_{(\text{xv})}\}$
 are all of central type and therefore can
 be in relation only with themselves.
 However, by \cite[Table 2]{Higgs2006} these groups admit non-isomorphic Schur multipliers and therefore they are all singletons.
 The rest follows from the computation of the complex group algebras over these groups (see \cite[p.4]{Higgs2006}), by the computation of the Schur multipliers (see \cite[Table 2]{Higgs2006})
 and by Lemma~\ref{lemma:p4enough2}.

 The $p=2$ case is done in a similar way, using \cite[Theorem 2]{Schnabelnew} or \cite[Table 3]{Higgs2006}
 for the classification of groups of central type of order $16$. Then the classification is completed using the description of the Schur multipliers
 of groups of order $16$ which can be found in \cite[Table 3]{Higgs2006}, by using the description of the complex group algebra over groups of order $16$ which can be found in
 \cite[p.7]{Higgs2006} and then applying Lemma~\ref{lemma:p4enough2}.
\end{proof}

We will now study the case where $G$ is a group of cardinality $p^2q^2$ for primes $p<q$.
It is well known (see e.g. \cite[Lemma 3.9]{ginosargradings})
that for $pq\neq 6$ the Sylow $q$-subgroup of $G$ is normal and therefore $G=Q\rtimes P$ where $P$ is a Sylow $p$-subgroup of $G$ and $Q$ is a Sylow $q$-subgroup of $G$.
In the case where $pq=6$ either $G=Q\rtimes P$ or $G=P\rtimes Q$.

\begin{theorem}\label{th:pqcase}
Let $G$ and $H$ be groups of cardinality $p^2q^2$ for primes $p<q$.
 If $M(G) \cong  M(H)$ and $\C G\cong \C H$, then $G\sim _{\C} H$.
 \end{theorem}
 \begin{proof}
 First, by Lemma~\ref{lemma:abcase} the theorem is clear for any abelian groups.
 Notice that if $G\cong Q\rtimes P$ and $H\cong P\rtimes Q$ are not abelian then $\C G$ and $\C H$ are not isomorphic.
 Therefore we may always assume that $G$ and $H$ admit the same decomposition as semi-direct products with respect to Sylow subgroups.
 From now assume that $G$ is a non-abelian group which admits a decomposition as $G=Q\rtimes P$.

 By Lemma \ref{eq:cohosemi}
\begin{equation}\label{eq:semdirectoverc}
 M(G) = M(Q)^P \times M(P).
\end{equation}
Therefore, by Corollary~\ref{cor:SchurofCn} the Schur multiplier is either trivial or isomorphic to $C_p,C_q$ or $C_{pq}$.
Now, by \cite[Theorem 3.21]{ginosargradings}, there exists at most one non-abelian group of the form $G=Q\rtimes P$ such that $M(G)$ is isomorphic to $C_{pq}$ and therefore
the theorem is obvious for this group. Notice also that the theorem is clear for all groups which have trivial Schur multiplier.
We are left with the cases in which $M(G)$ is isomorphic to either $C_p$ or to $C_q$.
Assume that $M(G) \cong C_q$. Then, since $p<q$, by Lemma~\ref{lemma:cohdiv} if $[\alpha]\in M(G)$ is a non-trivial cohomology class then
\begin{equation*}
 \C ^\alpha G\cong \oplus _{i=1}^{p^2} \C^{q \times q}.
\end{equation*}
Consequently, if $G$ and $H$ are groups such that $M(G) \cong  M(H) \cong C_q$ and $\C G\cong \C H$ then any isomorphism
from $M(G)$ to $M(H)$ induces a family of isomorphisms between the corresponding twisted group algebras and therefore $G\sim _{\C} H$.
Assume now that $M(G) \cong C_p$. In this case, $P\cong C_p\times C_p\cong \langle x\rangle \times \langle y\rangle$ and any non-trivial cohomology class $[\alpha]\in M(G)$
is determined by $[u_x,u_y]=\zeta _p$ in $\C ^\alpha G$ where $\zeta _p$ is a primitive $p$-th root of unity.
Therefore, a simple calculation shows that the set of the $\alpha$-regular classes in $G$ is the set of conjugacy classes contained in $Q$.
Now, since $\C G\cong \C H$ the number of conjugacy classes which is contained in the Sylow $q$-subgroup of $G$ and $H$ is the same.
Therefore, by the discussion at the end Section \ref{s:pre}
\begin{equation}\label{eq:samecenter}
 \text{dim}\left(Z\left(\C^\alpha G \right)\right)=\text{dim}\left(Z\left(\C ^{\beta}H \right)\right)
\end{equation}
for any non-trivial $[\alpha]\in M(G)$ and non-trivial $[\beta]\in M(H)$.
Since in the Artin-Wedderburn decomposition of $\C ^\alpha G$ and $\C ^{\beta}H $ a simple component is isomorphic to either $\C^{p \times p}$ or $\C^{p^2 \times p^2}$, and since $|G|=|H|$, by~\eqref{eq:samecenter}
$\C ^\alpha G\cong \C ^{\beta}H$.
Consequently, since also $\C G\cong \C H$ we get that $G\sim_{\C} H$.
The case where $p=2,q=3$ and $G$ is non-abelian of the form $G=P\rtimes Q$ is done in a similar way.
 \end{proof}

In the following example we show how to use the above theorem in order to classify all the $\sim_{\C}$-equivalence classes of $\Omega _n$ for $n=p^2q^2$ where $p<q$ are primes.
A key point in dealing with these cases is that we know how to describe all the groups of order $p^2q^2$ \cite{Vavasseur}.
We use the fact that $\operatorname{Aut}(C_q\times C_q)$ can be identified with ${\operatorname{GL}}_2(\mathbb{F}_q)$. We will take the more complicated case in which $p^2$ is a divisor of $q-1$ and in particular it will also divide the
cardinality of ${\operatorname{GL}}_2(\mathbb{F}_q)$. We denote a diagonal matrix $A=(a_{ij})\in {\operatorname{GL}}_2(\mathbb{F}_q)$ by $d(a_{11},a_{22})$.

\begin{example}\label{ex:319}
Let $p=3,q=19$ and let $n=p^2q^2$.
We will classify the $\sim_{\C}$-equivalence classes of $\Omega _n$.
First, any $G \in\Omega _n$ admits a decomposition as $G=Q\rtimes P$ where $Q$ and $P$ denote a Sylow $q$-subgroup and Sylow $p$-subgroup respectively.
By Corollary~\ref{cor:SchurofCn} and~\eqref{eq:semdirectoverc}, $|M(G)|$ is divisible by $p$ if and only if $P\cong C_p\times C_p$. On the other hand
$|M(G)|$ is divisible by $q$ if and only if $Q\cong C_q \times C_q$ and the action of $P$ on $Q$ is given by an element in ${\operatorname{SL}}_2(\mathbb{F}_q)$ (see \cite[Lemma 3.16]{ginosargradings}).
Denote the abelian groups
$$C_{q^2p^2},C_{q^2}\times C_p\times C_p,C_q\times C_q\times C_{p^2},C_q\times C_q\times C_p\times C_p$$
 by $G_1,G_2,G_3,G_4$ respectively.
We are now ready to start partitioning the groups of order $n$ by their Schur multiplier.
\begin{enumerate}
 \item The only abelian group of order $n$ with trivial Schur multiplier is $G_1$.
 There are $10$ non-abelian groups of order $n$ whose Schur multiplier is trivial.
 These groups are the two groups in which $Q$ and $P$ are both cyclic where in one of them, $G_5$, the action is of order $p$ and in the other, $G_6$, the action is of order $p^2$.
 The other $8$ non-abelian groups have the form $(C_q\times C_q)\rtimes C_{p^2}$ where the $8$ different actions are given by the following diagonal matrices:
 $$d(1,\zeta ^3),d(\zeta ^3,\zeta ^3),d(1,\zeta), d(\zeta,\zeta ),d(\zeta,\zeta ^2),d(\zeta,\zeta^3),d(\zeta,\zeta^4),d(\zeta,\zeta^6).$$
 Here $\zeta\in \mathbb{F}_q^*$ has order $p^2$.
 We number these groups as $G_7-G_{14}$.
 \item The only abelian group of order $n$ with Schur multiplier isomorphic to $C_p$ is $G_2$.
 There are $4$ non-abelian groups of order $n$ whose Schur multiplier is isomorphic to $C_p$.
 $G_{15}\cong C_{q^2}\rtimes (C_p\times C_p)$, $G_{16}\cong (C_q\times C_q)\rtimes (C_p\times C_p)$ where in $G_{16}$ the action is faithful.
 And there are another two groups, $G_{17},G_{18}$ of the form $(C_q\times C_q)\rtimes (C_p\times C_p)$ with non-faithful actions induced by the matrices
 $d(1,\zeta ^3),d(\zeta ^3,\zeta ^3)$.
 \item The only abelian group of order $n$ whose Schur multiplier is isomorphic to $C_q$ is $G_3$.
 There are two non-abelian groups of order $n$ whose Schur multiplier is isomorphic to $C_q$.
 Both these groups, $G_{19},G_{20}$ have the form $(C_q\times C_q)\rtimes C_{p^2}$ with corresponding actions given by the matrices  $d(\zeta ^3,\zeta ^6),d(\zeta ,\zeta ^8)$.
 \item The only abelian group of order $n$ whose Schur multiplier is isomorphic to $C_{pq}$ is $G_4$.
 Finally, there is a unique non-abelian group of order $n$ which admits a Schur multiplier isomorphic to $C_{pq}$.
 This group, $G_{21}$, is $(C_q\times C_q)\rtimes (C_p\times C_p)$ with a non-faithful action which is induced by the matrix $d(\zeta ^3,\zeta ^6)$.
\end{enumerate}
We sum this information in Table \ref{tab:SchurMult}.
   \begin{table}[h]
        \centering
       \caption{Schur multipliers of groups of order $3^2 \cdot 19^2$.} 
        \setlength{\tabcolsep}{2pt}
        \begin{tabular}{l r}
 \hline\hline 
$M(G)$ is trivial & $G_1$ or $G_i$ for $5\leq i\leq 14$\\
\hline
$M(G) \cong C_p$ & $G_2$ or $G_i$ for $15\leq i\leq 18$ \\
\hline
$M(G) \cong C_q$ & $G_3$ or $G_i$ for $19\leq i\leq 20$ \\
\hline
$M(G) \cong C_{pq}$ & $G_4$ or $G_{21}$\\
\hline
        \end{tabular}
          \label{tab:SchurMult}
    \end{table}

In order to apply Theorem~\ref{th:pqcase}, we now need to find the Artin-Wedderburn decomposition of $\C G_i$ for $1\leq i \leq 21$.
We will demonstrate the method to do so for $G_{12}$ and then will sum this information in Table \ref{tab:GroupRings}.
Let $G = (C_q\times C_q) \rtimes C_{p^2}$ where $p,q$ are primes such that $p^2$ divides $q-1$ and the action is given by $d(\zeta,\zeta ^p)$.
$G'\cong C_q\times C_q$ admits a trivial representation and $q^2-1$ representations of order $q$.
Consider the action of $C_{p^2}$ on the representations of $G'$.
Clearly the $G'$-trivial representation is stable under this action and hence
there are $p^2$ irreducible $1$-dimensional representations of $G$ which correspond to the trivial representation of $G'$.
Now, if a representation of $G'$ of order $q$ admits an orbit of length $p$ then it also admits a stabilizer of order $p$.
Therefore, for any such orbit of length $p$ of representations of $G'$ of order $q$ there are $p$ irreducible representations of $G$ of dimension $p$. Since the action is given by $d(\zeta,\zeta ^p)$, there are exactly $q-1$ irreducible representations of $G'$ with orbit length $p$.
We conclude that there are $q-1$ irreducible representations of $G$ of dimension $p$.
Finally, there are $q(q-1)$ irreducible representations of $G'$ with orbit length $p^2$.
To each such orbit there is a corresponding representation of $G$ of dimension $p^2$.
Consequently, there are $\frac{q(q-1)}{p^2}$ irreducible representations of $G$ of dimension $p^2$.
As a result
$$\C G\cong \C^{p^2} \oplus (q-1)\C^{p \times p} \oplus \frac{q(q-1)}{p^2} \C^{p^2 \times p^2}.$$
Which means that
$$\C G_{12}\cong 9\C \oplus 18\C^{3 \times 3} \oplus 38 \C^{9 \times 9}.$$
Similarly we can find the Artin-Wedderburn decomposition of the group algebras $\C G_i$ for any	 $1\leq i\leq 21$.
These can be found in Table \ref{tab:GroupRings}.
    \begin{table}[h]
        \centering
       \caption{Complex group algebras of groups of order $3^2 \cdot 19^2$.} 
        \setlength{\tabcolsep}{2pt}
        \begin{tabular}{l r}
 \hline\hline 
$\C G_i$ is abelian & $G_1,G_2,G_3,G_4$  \\
\hline
$\C G_i\cong 171\C \oplus 342\C^{3 \times 3}$ & $G_7,G_{17}$   \\
\hline
$\C G_i\cong 171\C \oplus 38 \C^{9 \times 9}$ & $G_9$   \\
\hline
$\C G_i\cong 9\C \oplus 360\C^{3 \times 3}$ & $G_5,G_8,G_{15},G_{18},G_{19},G_{21}$   \\
\hline
$\C G_i\cong 9\C \oplus 40 \C^{9 \times 9}$ & $G_6,G_{10},G{11},G_{13},G_{20}$   \\
\hline
$\C G_i\cong 9\C \oplus 18\C^{3 \times 3} \oplus 38\C^{9 \times 9}$ & $G_{12},G_{14}$   \\
\hline
$\C G_i\cong 9\C \oplus 36\C^{3 \times 3} \oplus 36\C^{9 \times 9}$ & $G_{16}$   \\
\hline
        \end{tabular}
          \label{tab:GroupRings}
    \end{table}

We conclude from both of the above tables and Theorem \ref{th:pqcase} that the $\sim_{\C}$-equivalence classes of $\Omega _{3^2 19^2}$ are as follows:
The groups $G_1$, $G_2$, $G_3$, $G_4$, $G_7$, $G_9$, $G_{16}$, $G_{17}$, $G_{19}$, $G_{20}$, $G_{21}$ are all singletons.
We have three equivalence classes containing two groups, $\{G_5,G_8\}$, $\{G_{12},G_{14}\}$, $\{G_{15},G_{18}\}$.
And finally, we have one equivalence class containing four groups, $\{G_6,G_{10}, G_{11},G_{13}\}$.
\end{example}

\section{Groups of central type}\label{se:groupsofct}
By Theorem~\ref{th:p4class} $p$-groups of central type of order $p^4$ are all $\sim _{\C}$ singletons.
Also for primes $p<q$, by \cite[Theorem 3.21]{ginosargradings} groups of central type of order $p^2q^2$ are all\\
$\sim _{\C}$ singletons.
In the cases where $q>3$ there is a unique non-abelian group of central type of order $p^2q^2$ while in the case $pq=6$ there are two
non-abelian groups of central type. However, since these groups admits commutator subgroups of different cardinality the complex group algebras over these groups are not isomorphic and therefore they are not $\sim _{\C}$-equivalent.
This gives the idea that a group of central type is always a $\sim_{\C}$-singleton.
We give some cases of cardinalities of groups for which this is true but eventually we give an example of two non-isomorphic groups of central type which are $\sim_{\C}$-equivalent. We remark however that we do not know a $p$-group of central type which is not a $\sim_\mathbb{C}$-singleton.

In order to deal with groups of central type of order $n^2$ where $n$ is square-free we record the following theorem from \cite{ginosargradings} which classifies these groups.
\begin{theorem}\label{th:squarefreedec}(see [\cite[Theorem A]{ginosargradings})
Let $n$ be a square-free number and let $G$ be a group of central type of order $n^2$.
\begin{enumerate}
 \item There exist $m,k\in \mathbb{N}$ such that $G'\cong C_m\times C_m$ and
$$G\cong G' \rtimes (C_k\times C_k),$$
where the action is given by an element in SL$_2(m)$.
\item The identity of $G'$ is the only element fixed under the $C_k\times C_k$ action.
\item For primes $p$, $q$ which divide $k$, $m$ respectively, any non-abelian subgroup of $G$ of cardinality $p^2q^2$ (if it exists) admits a center of cardinality $p$.
\end{enumerate}
 \end{theorem}
 On the other hand, groups of central type can be obtained in the following way.
 \begin{theorem}\label{th:ctotherdirection}\cite[Theorem 3.17]{ginosargradings}
  Let $K$ be a group of central type of cardinality prime to $m$ which acts on $C_m\times C_m$ via an element $\varphi \in \operatorname{SL}_2(m)$. Then
  $$(C_m\times C_m)\rtimes _{\varphi} K$$
  is of central type.
 \end{theorem}
 The following lemma will be the key in the proof of both cases of Theorem~\ref{th:centyp}.
\begin{lemma}\label{lemma:equictgr}
 Let $G$ and $H$ be groups of central type of cardinality $n^2$ where $n$ is square-free such that $\C G\cong \C H$.
 Then $G'\cong H'$ and $Z(G)\cong Z(H)$.
\end{lemma}
\begin{proof}
Recall that the number of $1$-dimensional representations of a group $G$ is equal to the cardinality of $G/G'$.
Therefore, $\C G\cong \C H$ implies that $|G'|=|H'|$ and consequently by Theorem~\ref{th:squarefreedec} there exists $m\in \mathbb{N}$ such that
$G'\cong H'\cong C_m\times C_m.$
Now let
$$G=\left(C_m\times C_m\right)\rtimes_{\psi} (C_k\times C_k), \ \ H\cong \left(C_m\times C_m\right)\rtimes_{\varphi} ( C_k\times C_k),$$
and let $p$ be a prime divisor of $k$.
By Theorem~\ref{th:squarefreedec}(2), $Z(G)=\text{ker}(\psi)$ and $Z(H)=\text{ker}  (\varphi)$.
Therefore, by showing $\text{ker}  (\psi)\cong \text{ker}  (\varphi)$ we prove $Z(G)\cong Z(H)$.
If ${\operatorname{ker}}(\psi)$ contains a subgroup isomorphic to $C_p\times C_p$ then there are no irreducible representation of $G$ whose dimension is divisible by $p$.
Since, $\C G\cong \C H$, we know that if ${\operatorname{ker}}(\psi)$ contains a subgroup isomorphic to $C_p\times C_p$ then so does ${\operatorname{ker}}(\varphi)$.
Assume now that $p$ is not a divisor of ${\operatorname{ker}}(\psi)$.
By Theorem~\ref{th:squarefreedec} there exist primes $q$, $r$ which divide $m$ such that a Hall $pqr$-subgroup of $G$ is of the form
$$ \left(C_{qr}\times C_{qr}\right)\rtimes (C_p\times C_p) $$
and admits a trivial center.
In particular, a representation of $G'$ of dimension $(qr)^2$ has a $\psi$-orbit of length divisible by $p^2$ and therefore $G$ admits an irreducible representation of dimension divisible by $p^2$.
Since $\C G\cong \C H$, it follows that $H$ also admits an irreducible representation of
dimension divisible by $p^2$ and consequently, by similar considerations $p$ is not a divisor of ${\operatorname{ker}}(\varphi)$ and hence $\text{ker}  (\psi)\cong \text{ker}  (\varphi)$.
\end{proof}
We record the following observation.
\begin{remark}\label{re:abaction}
Let $K$ be a nilpotent group acting on the groups $N_1,N_2$ and let $G=N_1\rtimes K, H=N_2\rtimes K$ be the groups corresponding to these actions.
Then, $G\cong H$ if and only if $N_1\rtimes P \cong N_2\rtimes P$ for any Sylow subgroup $P$ of $K$.
\end{remark}
Using this remark we prove in Proposition \ref{prop:cyclicaction1} and Proposition \ref{prop:atmosttwoprimes} that under some conditions groups of central type are $\sim_\mathbb{C}$-singletons.
\begin{proposition}\label{prop:cyclicaction1}
Let $G$ be a group of central type of cardinality $n^2$ such that $n$ is square-free.
If $\frac{|G|}{|G'||Z(G)|}$ is square-free then $G$ is a $\sim _{\C}$-singleton.
\end{proposition}
\begin{proof}
Let $G, H \in \Omega _{n^2}$ such that $G\sim _{\C}H$.
By Theorem~\ref{th:squarefreedec} and Lemma~\ref{lemma:equictgr} there exist $m,k$ such that
$$G\cong \left(C_m\times C_m\right)\rtimes _{\varphi} (C_k\times C_k) \text{ and } H\cong \left(C_m\times C_m\right)\rtimes _{\psi} (C_k\times C_k).$$
We will show that for any $p$ which divides $k$
$$\tilde{G}=\left(C_m\times C_m\right)\rtimes _{\varphi} (C_p\times C_p) \cong \tilde{H}=\left(C_m\times C_m\right)\rtimes _{\psi} (C_p\times C_p).$$
If this is the case, then $G\cong H$ by Remark~\ref{re:abaction}.
Now, by Lemma~\ref{lemma:equictgr}, if $Z(G)$ contains a subgroup isomorphic to $C_p \times C_p$, then also $Z(H)$ contains a subgroup isomorphic to $C_p\times C_p$. So in this case it is clear that $\tilde{G}\cong \tilde{H}$.
Assume that $p^2$ is not a divisor of $|Z(\tilde{G})|$, then clearly $p^2$ is also not a divisor of $|Z(\tilde{H})|$.
By the condition of the proposition $p$ is a divisor of $|Z(\tilde{G})|$ and $|Z(\tilde{H})|$ and therefore the action in both cases is a $C_p$ -action.
Applying \cite[Lemma 3.20 $(ii)$]{ginosargradings} we get that $\tilde{G}\cong \tilde{H}$ which completes the proof.
\end{proof}
\begin{proposition}\label{prop:atmosttwoprimes}
Let $G$ be a group of central type of order $n^2$ where $n$ is a square-free number such that $|G'|$ is divisible by at most two primes.\\
 Then $G$ is a $\sim _{\C}$-singleton.
\end{proposition}
\begin{proof}
Assume $G\sim _{\C} H$.
By Lemma~\ref{lemma:equictgr}, $G'\cong H'$ and therefore by Theorem~\ref{th:squarefreedec} there exist $m,k\in \mathbb{N}$ such that
 $G$ and $H$ both admit a decomposition as
 $$\left(C_m\times C_m\right)\rtimes (C_k\times C_k),$$
 where $m$ is divisible by at most two primes.
 By Remark~\ref{re:abaction} it is enough to prove the theorem for the case $k = p$ a prime.
 So we may assume
 $$G=\left(C_m\times C_m\right)\rtimes _{\varphi} (C_p\times C_p), \ \  H=\left(C_m\times C_m\right)\rtimes _{\psi} (C_p\times C_p).$$
 By Lemma~\ref{lemma:equictgr} and Proposition~\ref{prop:cyclicaction1} the theorem holds in the case where $G$ (and therefore also $H$)
 admits a non-trivial center, and therefore we may assume that $Z(G)$ and $Z(H)$ are trivial.
 This settles also the case in which $m$ is prime, since by \cite[Theorem 3.21]{ginosargradings}, if $n$ is the product of two primes, then the center of $G$ (and $H$) is non-trivial.
 We are left with the case in which $m$ is the product of two primes $q_1\neq q_2$.
 By \cite[Theorem 3.21]{ginosargradings} the restriction of $\varphi$ (similarly of $\psi$) to the normal $q_1$- or $q_2$-subgroups have a cyclic kernel.
 Denote
 $$\operatorname{ker}(\operatorname{res}(\varphi)_{q_1})=\langle x\rangle,   \operatorname{ker}(\operatorname{res}(\varphi)_{q_2})=\langle y\rangle,
\operatorname{ker}(\operatorname{res}(\psi)_{q_1})=\langle z\rangle, \operatorname{ker}(\operatorname{res}(\psi)_{q_2})=\langle w\rangle.$$
 Since $G$ and $H$ admit a trivial center,
 $$\langle x\rangle\times \langle y\rangle\cong C_p\times C_p\cong \langle z\rangle\times \langle w\rangle.$$
 Consequently, the $C_p\times C_p$ isomorphism which maps $x$ to $z$ and $y$ to $w$ induces an isomorphism from $G$ to $H$.
\end{proof}
Next we wish to present an example of two non-isomorphic groups $G$ and $H$ of central type of order $n^2$, where $n$ is a square-free number, such that $G\sim _{\C} H$. In order to do so we first need to understand a method to evaluate dimensions of projective representations as described in \cite[Ch. 5, Sec. 2]{karpilovsky2}.

Let $m$ be a square-free number, $p$ a prime which is not a divisor of $m$ and let $G$ be a group of central type of order $m^2p^2$ such that
$G'\cong C_m\times C_m = \langle x,y \rangle$ and the action of $C_p\times C_p = \langle a \rangle \times \langle b \rangle$ is given by an element in ${\operatorname{SL}}_2(m)$.
Roughly speaking, a cohomology class $[\alpha]\in M(G)$ is determined by the relations
$$[u_x,u_y]=\zeta _m,[u_a,u_b]=\zeta _p,$$
where $\zeta _m,\zeta _p$ are, not necessarily primitive, $m$-th and $p$-th roots of unity, respectively.
Assume that $\zeta _m$ has order $d$, then
$$\C ^\alpha G' \cong \oplus _{i=1}^{(\frac{m}{d})^2} \C^{d \times d},$$
where here we call the restriction of $\alpha$ to $G'$ also $\alpha$.
We wish to determine the Artin-Wedderburn decomposition of $\C^\alpha G$ depending on whether $\zeta _p$ is trivial or not.
Notice, that in the case where $\zeta _p$ is of order $p$, by Lemma~\ref{lemma:cohdiv}, the dimension of each $\alpha$-projective representation is divisible by $p$.
Hence, in this case, by simple dimension considerations
$$\C ^\alpha G\cong \oplus _{i=1}^{(\frac{m}{d})^2} \C^{dp \times dp}.$$
On the other hand, if $\zeta _p$ is trivial, then
$$\C ^\alpha G\cong \C^{d \times d}\otimes \C H,$$
where $H\cong \left(C_{\frac{m}{d}}\times C_{\frac{m}{d}}\right)\rtimes (C_p\times C_p)$ is a subgroup of $G$.
By similar considerations the following holds.
\begin{proposition}\label{prop:projdim}
 Let $G$ be a group of central type of order $n^2$ where $n$ is a square-free number with a corresponding decomposition as in Theorem~\ref{th:squarefreedec}(1).
 Denote by $H$ the subgroup of $G$ isomorphic to $\left(C_{\frac{m}{d}}\times C_{\frac{m}{d}}\right)\rtimes (C_k\times C_k)$ where $d$ is a divisor of $m$.
 For any $[\alpha]\in M(G)$ such that ${\operatorname{res}}^G _{G'}[\alpha]$ has order $d$ the following hold
 \begin{enumerate}
  \item $$\C ^\alpha G\cong \C^{d \times d}\otimes \C ^\alpha H.$$
  \item If $d=1$ (that is ${\operatorname{res}}^G _{C_m\times C_m}[\alpha]$ is trivial), $k=p$ is prime and ${\operatorname{res}}^G _{C_k\times C_k}[\alpha]$ is non-trivial then
  if $\C G\cong A\oplus \left( \oplus_{i=1}^{p^2} \mathbb{C}\right)$ then $\C ^\alpha G\cong A\oplus \C^{p \times p}$.
  \end{enumerate}
\end{proposition}
The following example presents two non-isomorphic groups $G$ and $H$ of central type of order $n^2$ where $n$ is a square-free number such that $G\sim _{\C} H$. Here $|G'|=|H'|$ is divisible by three primes.

\begin{example}\label{ex:3primes}
The goal of this example is to present two non-isomorphic groups $G$ and $H$ of central type satisfying $G \sim _{\mathbb{C}}H$. It has been constructed together with Yuval Ginosar who we would like to thank at this place.

We fix a prime $q \geq 5$ and let $p_1$, $p_2$ and $p_3$ be three primes such that $\SL(2,p_i)$ contains an element of order $q$ for each $1 \leq i \leq 3$. Let $K= \langle x, y \rangle \cong C_q \times C_q$ and $K' = \langle x', y' \rangle \cong C_q \times C_q$. Let $N \cong C_{p_1 p_2 p_3} \times C_{p_1 p_2 p_3}$, i.e. the direct product of three elementary abelian groups of rank $2$. We define two groups
\[ G = N\rtimes _{\varphi_1}K, \ \ H = N\rtimes _{\varphi_2}K' \]
by specifying an action of $K$ and $K'$ on the Sylow subgroups of $N$, respectively. $K$ and $K'$ act on the Sylow $p_1$-subgroup with kernel $\langle y \rangle$ and $\langle y' \rangle$ and with $x$ and $x'$ acting by the same matrix of order $q$ in $\SL(2,p_1)$, respectively. Similarly, $K$ and $K'$ act on the Sylow $p_2$-subgroup of $N$ with kernel $\langle x \rangle$ and $\langle x' \rangle$ and with $y$ and $y'$ acting by the same matrix of order $q$ in $\SL(2,p_2)$, respectively. Finally, $x$ and $x'$ act by the same matrix of order $q$ from $\SL(2,p_3)$ on the Sylow $p_3$-subgroup of $N$ and the kernel of the action of $K$ is $\langle xy \rangle$ while the kernel of the action of $K'$ is $\langle x^2y \rangle$.

These two groups are of central type by \cite[Theorem 3.17]{ginosargradings} (also appearing above as Theorem~\ref{th:ctotherdirection}). We will first verify that $G \sim_\mathbb{C} H$. 
By Lemma~\ref{eq:cohosemi} and Corollary~\ref{cor:SchurofCn}, we have
\[M(G) \cong M(H )\cong C_{q \cdot p_1 \cdot p_2 \cdot p_3}.\]
Next, we show that these groups admit isomorphic complex group algebras. As for both groups the maximal abelian quotient is isomorphic to $K$, both $G$ and $H$ admit $q^2$ non-equivalent 1-dimensional representations. Next, let $\chi$ be a non-trivial irreducible representation of $G'$ or $H'$. Note that $\chi$ is $1$-dimensional and so equals its character.
Then, $\chi$ admits a non-trivial stabilizer if and only if for two different primes $p_\alpha$ and $p_\beta$ the Sylow $p_\alpha$-subgroup and Sylow $p_\beta$-subgroup are contained in the kernel of $\chi$, i.e. if $\chi(n)^p = 1$ for every $n \in N$ and some fixed prime $p$.
In these cases $\chi$ admits a stabilizer of order $q$ and lies in a $K$-orbit and $K'$-orbit of length $q$, respectively.
Since for both $G'$ and in $H'$ there are $p_i^2-1$ irreducible non-trivial representations $\chi$ with $\chi(n)^{p_i} = 1$, for each $1 \leq i \leq 3$, there are $(p_1^2 + p_2^2 + p_3^2 - 3)$ irreducible representations of $G$ of degree $q$ and exactly as many of $H$.
Since $N$ is abelian and normal of index $q^2$ in both $G$ and $H$, all the irreducible representations of $G$ and $H$ which admit a degree greater than $q$ have degree $q^2$ by It\^{o}'s Theorem (cf. \cite[Theorem 6.15]{Isaacs}). Hence, by a simple dimension consideration $\C G\cong \C H$.

Finally, for any proper divisor $d$ of $\sqrt{|G|}$, the number of $d$-Hall subgroups of $G$ is equal to the number of $d$-Hall subgroups of $H$, and moreover any $d$-Hall subgroup $G_1$ of $G$ is isomorphic to any $d$-Hall subgroup $H_1$ of $H$. Therefore, for any cohomology class $[\alpha]\in M(G)$ determined by a $d$-th root of unity $\zeta _d$, there is $[\beta] \in M(H)$ determined by $\zeta _d$ such that $\C^{\alpha}G_1 \cong \C^{\beta} H_1$.
Therefore, by Proposition~\ref{prop:projdim} for any isomorphism $\psi: M(G) \rightarrow M(H)$ we have $\C ^{\alpha} G \cong \C^{\psi(\alpha)}H$ for any $[\alpha]\in M(G)$. Hence $G \sim_\C H$.

We next show that $G$ and $H$ are not isomorphic. Assume to the contrary that $\psi: G \rightarrow H$ is an isomorphism. Clearly $N$ is a characteristic subgroup in both $G$ and $H$ and so $\psi(N) = N$. As any two complements of $N$ in $H$ are conjugate by the Schur-Zassenhaus Theorem, we can also assume $\psi(K) = K'$. For $1 \leq i \leq 3$, denote by $N_i$ the Sylow $p_i$-subgroup of $N$. Let $A_1$ be the matrix by which $x$ acts on $N_1$. Then for any $n_1 \in N_1$ we have
\[ \psi(n_1)\psi(x) = \psi(n_1x) = \psi(x n_1^x) = \psi(x) A_1 \psi(n_1),\]
so $\psi(n_1)^{\psi(x)} = A_1 \psi(n_1)$, i.e. the actions of $x$ and $\psi(x)$ are conjugate in $\SL(2,p_1)$. As $K'$ acts on $N_1$ also by $\langle A_1 \rangle$ and the only matrices in $\langle A_1 \rangle$ conjugate to $A_1$ are $A_1^{\pm 1}$, we conclude $\psi(x) = (x')^{\pm 1} (y')^{j_x}$ for some $j_x$. With an analogues calculation, as $y$ acts trivially on $N_1$ and the kernel of the $K'$-action on $N_1$ is $\langle y' \rangle$, we obtain $\psi(y)  \in \langle y' \rangle$. Repeating the calculations for the action on $N_2$, we get $\psi(x) \in \langle x' \rangle$ and $\psi(y) = (y')^{\pm 1}(x')^{j_y}$ for some $j_y$. So from both actions we get $\psi(x) = (x')^{\pm 1}$ and $\psi(y) = (y')^{\pm 1}$. But considering the action on $N_3$ and using that the kernel of the action of $K$ is $\langle xy \rangle$ and the kernel of the action of $K'$ is $\langle (x')^2 y' \rangle$ we get $\psi(xy) \in \langle (x')^2 y' \rangle$, contradicting the conditions on $\psi(x)$ and $\psi(y)$ found before, as we assumed $q \geq 5$.
\end{example}

\begin{remark}\label{Order64}
We are not aware of non-isomorphic $p$-groups of central type which are $\sim_\mathbb{C}$-equivalent.
Using the computer algebra system \texttt{GAP} we found that every group of order $64$ which is of central type is a $\sim_\mathbb{C}$-singleton. We provide a few details.
The \texttt{GAP}-functions \texttt{CharacterDegrees},
\texttt{AbelianInvariantsMultiplier} and \texttt{SchurCover} make it possible to check conditions A, B and C introduced in Definition \ref{def:conditions} for a given pair of groups.
By determining the maximal degree of a complex representation of a Schur cover of a group it also allows to check, whether a group is of central type.
Since \texttt{GAP} computes the Schur cover as a finitely presented, rather than a polycyclic group,
this may however take a while or not even terminate in a reasonable time frame.
Applying also the theoretical arguments, \cite[Theorem C]{ginosar2012semi} and \cite[Lemma 2.3, Lemma 2.4]{Schnabelnew},
for a few groups, we found that there are $42$ groups of central type of order $64$.
There are five pairs of these groups which satisfy conditions A and B, but none of these pairs satisfies condition C.
So indeed any group of order $64$ of central type is a $\sim_\mathbb{C}$-singleton.
\end{remark}

{\bf Acknowledgements.}
The authors are grateful to Y. Ginosar and W. Kimmerle for valuable discussions.

\begin{appendices}
	\center{\textbf{Appendix}}
    \begin{table}[htb!]
        \centering
       \caption{Groups of order $p^4$ where $p$ is an odd prime} 
        \setlength{\tabcolsep}{2pt}
        \resizebox{\linewidth}{!}{%
        \begin{tabular}{l c l}
 \hline\hline 
\\ 
\hline
$G_{(\text{i})}$ & $C_{p^4}$ &    \\
\hline
$G_{(\text{ii})}$ & $C_{p^3}\times C_p$ &    \\
\hline
$G_{(\text{iii})}$ & $C_{p^2}\times C_{p^2}$ &    \\
\hline
$G_{(\text{iv})}$ & $C_{p^2}\times C_p\times C_p$ &\\
\hline
$G_{(\text{v})}$ & $C_p^4$ &  \\
\hline
$G_{(\text{vi})}$ & $\langle a,b\rangle$ & $a^{p^3}=b^p=1,\quad aba^{-1}=ba^{p^2}$ \\
\hline
$G_{(\text{vii})}$ & $\langle a,b,c \rangle$ & $a^{p^2}=b^p=c^p=1,\quad  [a,b]=[a,c]=1,[b,c]=a^p$ \\
\hline
$G_{(\text{viii})}$ & $\langle a,b \rangle$ & $a^{p^2}=b^{p^2}=1,\quad  [a,b]=a^p$ \\
\hline
$G_{(\text{ix})}$ & $\langle a,b,c \rangle$ & $a^{p^2}=b^p=c^p=1,\quad  [a,b]=[b,c]=1, [a,c]=a^p$ \\
\hline
$G_{(\text{x})}$ & $\langle a,b,c \rangle$ & $a^{p^2}=b^p=c^p=1,\quad  [a,b]=[b,c]=1, [a,c]=b$ \\
\hline
$G_{(\text{xi})}$ for $p>3$ & $\langle a,b,c \rangle$ & $a^{p^2}=b^p=c^p=1,\quad  [b,c]=1, [a,b]=a^p,ac=cab$ \\
\hline
$G_{(\text{xii})}$ for $p>3$ & $\langle a,b,c \rangle$ & $a^{p^2}=b^p=c^p=1,\quad  [b,c]=a^p, [a,b]=a^p,ac=cab$ \\
\hline
$G_{(\text{xiii})}$ for $p>3$ & $\langle a,b,c \rangle$ & $a^{p^2}=b^p=c^p=1,\quad  [b,c]=a^{\alpha p}, [a,b]=a^p,ac=cab$, $\alpha=\text{any non-residue } (\text{mod } p)$ \\
\hline
$G_{(\text{xi})}$ for $p=3$ & $\langle a,b,c \rangle$ & $a^9=b^3=c^3=1,\quad  [b,c]=1, [a,b]=a^3,[a,c]=ba^3$ \\
\hline
$G_{(\text{xii})}$ for $p=3$ & $\langle a,b,c \rangle$ & $a^9=b^3=1,c^3=a^3,\quad  [b,c]=1, [a,b]=a^p,[c,a]=ba^3$ \\
\hline
$G_{(\text{xiii})}$ for $p=3$ & $\langle a,b,c \rangle$ & $a^9=b^3=1,c^3=a^6,\quad  [b,c]=1, [a,b]=a^p,[c,a]=ba^3$ \\
\hline
$G_{(\text{xiv})}$  & $\langle a,b,c,d \rangle$ & $a^p=b^p=c^p=d^p=1,\quad  [a,b]=[a,c]=[a,d]=[b,c]=[b,d]=1, [c,d]=a$ \\
\hline
$G_{(\text{xv})}$ for $p>3$ & $\langle a,b,c,d \rangle$ & $a^p=b^p=c^p=d^p=1,\quad  [a,b]=[a,c]=[a,d]=[b,c]=1,[d,b]=a, [d,c]=b$ \\
\hline
$G_{(\text{xv})}$ for $p=3$ & $\langle a,b,c \rangle$ & $a^9=b^3=c^3=1,\quad  [a,b]=1,[a,c]=ba^3,[b,c]=a^3$ \\
\hline
        \end{tabular}}
          \label{tab:p4}
    \end{table}

\vspace{2cm}

    \begin{table}[htb!]
        \centering
       \caption{Groups of order $16$} 
        \setlength{\tabcolsep}{2pt}
        \begin{tabular}{l c l}
 \hline\hline 
\\ 
\hline
$G_{(\text{i})}$ & $C_{16}$ &    \\
\hline
$G_{(\text{ii})}$ & $C_{8}\times C_2$ &    \\
\hline
$G_{(\text{iii})}$ & $C_{4}\times C_{4}$ &    \\
\hline
$G_{(\text{iv})}$ & $C_{4}\times C_2\times C_2$ &\\
\hline
$G_{(\text{v})}$ & $C_2^4$ &  \\
\hline
$G_{(\text{vi})}$ & $\langle a,b\rangle$ & $a^{8}=b^2=1,\quad aba^{-1}=ba^4$ \\
\hline
$G_{(\text{vii})}$ & $\langle a,b,c \rangle$ & $a^{4}=b^2=c^2=1,\quad  [a,b]=[a,c]=1,[b,c]=a^2$ \\
\hline
$G_{(\text{viii})}$ & $\langle a,b \rangle$ & $a^4=b^4=1,\quad  [a,b]=a^2$ \\
\hline
$G_{(\text{ix})}$ & $\langle a,b,c \rangle$ & $a^4=b^2=c^2=1,\quad  [a,b]=[b,c]=1, [a,c]=a^2$ \\
\hline
$G_{(\text{x})}$ & $\langle a,b,c \rangle$ & $a^4=b^2=c^2=1,\quad  [a,b]=[b,c]=1, [a,c]=b$ \\
\hline
$G_{(\text{xi})}$  & $\langle a,b,c \rangle$ & $a^4=c^2=1,a^2=b^2\quad  [a,c]=[b,c]=1, [a,b]=a^2$ \\
\hline
$G_{(\text{xii})}$ & $\langle a,b \rangle$ & $a^8=b^2=1,\quad  [a,b]=a^2$ \\
\hline
$G_{(\text{xiii})}$  & $\langle a,b \rangle$ & $a^8=b^2=1,\quad  [a,b]=a^6$ \\
\hline
$G_{(\text{xiv})}$  & $\langle a,b \rangle$ & $a^8=1, b^2=a^4 \quad  [a,b]=a^6$ \\
\hline
        \end{tabular}
          \label{tab:16}
    \end{table}
\end{appendices}


\begin{thebibliography}{10}

%



\bibitem{AGdR}
E.~Aljadeff, Y.~Ginosar, and {\'A}.~del R{\'{\i}}o.
\newblock Semisimple strongly graded rings.
\newblock {\em J. Algebra}, 256(1):111--125, 2002.


%


\bibitem{Brown2}
K.~S. Brown.
\newblock {\em Cohomology of groups}, volume~87 of {\em Graduate Texts in
  Mathematics}.
\newblock Springer-Verlag, New York, 1994.
\newblock Corrected reprint of the 1982 original.


\bibitem{MR913203}
R.~Brown, D.~L. Johnson, and E.~F. Robertson.
\newblock Some computations of nonabelian tensor products of groups.
\newblock {\em J. Algebra}, 111(1):177--202, 1987.







\bibitem{GAP}
The GAP~Group.
\newblock GAP -- Groups, Algorithms, and Programming, Version 4.8.5, 2016.
 \url{http://www.gap-system.org},


\bibitem{ginosar2012semi}
Y.~Ginosar and O.~Schnabel.
\newblock Semi-invariant {M}atrices over {F}inite {G}roups.
\newblock {\em Algebr. Represent. Theory}, 17(1):199--212, 2014.

\begingroup
\raggedright
\bibitem{ginosargradings}
Y.~Ginosar and O.~Schnabel.
\newblock Groups of central type, maximal connected gradings and intrinsic
  fundamental groups of complex semisimple algebras.
\newblock {\em Trans. Amer. Math. Soc.}, 371(9):6125-6168, 2019.
 \endgroup


\bibitem{Higgs2006}
R.~Higgs and D.~Healy.
\newblock Projective character degree patterns of groups of order {$p^4$}.
\newblock {\em Comm. Algebra}, 34(12):4623--4630, 2006.

\bibitem{Hertweck}
M.~Hertweck.
\newblock A counterexample to the isomorphism problem for integral group rings.
\newblock {\em Ann. of Math. (2)}, 154(1):115--138, 2001.


\bibitem{Higgs88}
R.~J. Higgs.
\newblock Groups with two projective characters.
\newblock {\em Math. Proc. Cambridge Philos. Soc.}, 103(1):5--14, 1988.

\bibitem{HoffmanHumphreys}
P. Hoffman and J.~F. Humphreys.
\newblock Non-isomorphic groups with the same projective character tables.
\newblock {\em Comm. Algebra}, 15(8):1637--1648, 1987.




\bibitem{Huppert}
B. Huppert,
\newblock{\em Endliche Gruppen. I}, Die Grundlehren der Mathematischen Wissenschaften, Band 134, Springer-Verlag, Berlin-
New York, 1967.

\bibitem{Isaacs}
I. M. Isaacs.
\newblock{\em Character theory of finite groups}, Academic Press
[Harcourt Brace Jovanovich, Publishers], New York-London,
1976,
\newblock{\em Pure and Applied Mathematics}, No. 69.


\bibitem{karpilovsky1}
G.~Karpilovsky.
\newblock {\em The {S}chur multiplier}, volume~2 of {\em London Mathematical
  Society Monographs. New Series}.
\newblock The Clarendon Press, Oxford University Press, New York, 1987.

\bibitem{karpilovsky2}
G.~Karpilovsky.
\newblock {\em Group representations. Vol. 2.}, {\em North-Holland Mathematics Studies, 177}.
\newblock North-Holland Publishing Co., Amsterdam, 1993. xvi+902 pp.


\bibitem{Vavasseur}
R.~Le~Vavasseur.
\newblock Les groupes d'ordre {$p^2q^2, p$} \'etant un nombre premier plus
  grand que le nombre premier {$q$}.
\newblock {\em Ann. Sci. \'Ecole Norm. Sup. (3)}, 19:335--355, 1902.

\bibitem{Oystaeyen}
E.~Nauwelaerts and F.~Van~Oystaeyen.
\newblock The {B}rauer splitting theorem and projective representations of
  finite groups over rings.
\newblock {\em J. Algebra}, 112(1):49--57, 1988.



\bibitem{Perlis}
S.~Perlis and G.~L. Walker.
\newblock Abelian group algebras of finite order.
\newblock {\em Trans. Amer. Math. Soc.}, 68:420--426, 1950.


\bibitem{RoggenkampScott}
K.~W.~Roggenkamp and L.~L.~Scott.
\newblock Isomorphisms of {$p$}-adic group rings.
\newblock {\em Ann. of Math. (2)}, 126(3):593--647, 1987.

\bibitem{Schnabelnew}
O.~Schnabel.
\newblock Simple twisted group algebras of dimension {$p^4$} and their
  semi-centers.
\newblock {\em Comm. Algebra}, 44(12):5395--5425, 2016.



\bibitem{Schur}
I.~Schur.
\newblock Untersuchungen \"{u}ber die {D}arstellung der endlichen {G}ruppen durch
  gebrochene lineare {S}ubstitutionen.
\newblock {\em Journal f\"{u}r die reine und angewandte Mathematik.},
  (132):85--137, 1907.

\bibitem{Tahara}
K.I.~Tahara
\newblock On the second cohomology groups of semidirect products
\newblock{\em Math. Z.}, 129, 365--379, 1972




%


\end{thebibliography}
\end{document}